\documentclass[11pt]{article}
\usepackage[utf8]{inputenc}
\usepackage[colorlinks=true, linkcolor=blue, urlcolor=blue, citecolor=blue]{hyperref}
\usepackage[
a4paper,
textwidth=17cm,
textheight=24cm,
top=1.8cm,
bottom=2.3cm,
centering
]{geometry}

\usepackage{amsfonts}
\usepackage{amsfonts,latexsym}
\usepackage{amsmath}
\usepackage{amsthm}
\usepackage{amssymb}
\usepackage{mathrsfs}
\usepackage{xcolor}
\usepackage{cite}
\usepackage{verbatim}
\usepackage{enumitem}
\usepackage{amsmath, amssymb, booktabs} 
\usepackage{array}  
\usepackage{xcolor}  
\usepackage{hyperref}  
\usepackage{xcolor}  
\usepackage{hyperref}  
\hypersetup{
	allcolors = blue  
}

\newtheorem{theorem}{\indent Theorem}[section]
\numberwithin{theorem}{section}  
\newtheorem{proposition}[theorem]{\indent Proposition}
\newtheorem{definition}[theorem]{\indent Definition}
\newtheorem{lemma}[theorem]{\indent Lemma}
\newtheorem{remark}[theorem]{\indent Remark}

\newtheorem{corollary}[theorem]{\indent Corollary}

\numberwithin{equation}{section}

\begin{document}
\title{Lagrangian chaos for the 2D Navier-Stokes equations driven by mildly degenerate noise}
\author{ Dengdi Chen$^{1}$, Yan Zheng \thanks{\emph{E-mail addresses:} yanzhengyl@163.com} \\
}
\date{}

\maketitle
\begin{abstract}

We consider the 2D incompressible Navier–Stokes equations driven by mildly degenerate noise that acts only on finitely many low Fourier modes, a setting that models large-scale stirring. For this system, we prove that the top Lyapunov exponent of the associated Lagrangian flow is strictly positive, thereby establishing Lagrangian chaos. This result is obtained within the framework of random dynamical systems, combining the multiplicative ergodic theorem with the refined Furstenberg criterion of \cite{NS}. Unlike the method in \cite{NS} for handling highly degenerate noise, this paper develops a unified analytical framework that combines low-mode control, finite-dimensional Malliavin calculus, and dissipation in the high modes. By constructing a finite-dimensional partial Malliavin matrix and proving its non-degeneracy, we avoid the technical complexity of performing Malliavin analysis on the full phase space and simultaneously overcome the degeneracy introduced by the manifold variables. Furthermore, the mildly degenerate forcing gives controllability in the low-frequency subsystem. In the manifold directions, only first-order Lie brackets are needed, which substantially simplifies the Lie-brackets computations.
\end{abstract}
\section{Introduction}
Chaotic behavior in dynamical systems has been studied extensively in recent decades, with important applications to fluid dynamics and turbulence \cite{V2-BJPV98}. Chaos is generally associated with sensitive dependence on initial conditions, topological transitivity, and dense periodic orbits. Among these, the study of sensitive dependence focuses on a core issue: whether a positive top Lyapunov exponent exists on a subset of the phase space with positive measure. This question remains a central concern for both mathematicians and physicists.

In this paper, we study the stochastic flow of diffeomorphisms $\mathbf{x}_t:\mathbb{T}^2\to \mathbb{T}^2$, $t\ge 0$, defined by the random ODE
\begin{equation} \label{Lagrange1}
	\frac{d}{dt}\mathbf{x}_t=u_t(\mathbf{x}_t),\quad \mathbf{x}_0=x.
\end{equation}
Here, the random velocity field $\mathbf{u}_t:\mathbb{T}^2\to \mathbb{R}^2$ at time $t\ge 0$ evolves according to the following incompressible stochastic Navier-Stokes equations
\begin{align} \label{2NS}
\begin{cases} 
\dot{u}_t = \nu \Delta u_t - u_t \cdot \nabla u_t + \sum_{k \in \mathcal{Z}_0} q_k \nabla^{\perp} \Delta^{-1} e_k(x) \, \Dot{W}^k_t, \\ 
\nabla \cdot u_t = 0, \\ 
u_0 \in \mathcal{H}, 
\end{cases}
\end{align}
where \( \mathcal{Z}_0 \subseteq \mathbb{Z}^2 \setminus \{0\} \) is the finite collection of forced modes, \(\{W^k_t\}_{k \in \mathcal{Z}_0}\) are a family of independent one-dimensional Wiener processes on a common canonical filtered probability space $(\Omega,\mathcal{F},(\mathcal{F}_t),\mathbb{P})$, \(\nu > 0\), \(q_k \neq 0\) for each \(k \in \mathcal{Z}_0\), \(\{e_k\}_{k \in \mathbb{Z}^2}\) is the real-valued Fourier basis defined in \eqref{ek}, and note that $\sum_{k \in \mathcal{Z}_0} q_k \nabla^{\perp} \Delta^{-1} e_k(x) \, W^k_t \in L^2_0(\mathbb{T}^2; \mathbb{R}^2)$. Here, \(\mathcal{H}:=H^{5}_{0}(\mathbb{T}^2; \mathbb{R}^2)\) and $L^2_0(\mathbb{T}^2; \mathbb{R}^2)$ denote the divergence-free, mean-zero vector fields in \(H^5(\mathbb{T}^2; \mathbb{R}^2)\) and \(L^2(\mathbb{T}^2; \mathbb{R}^2)\) respectively. The spatial variable $x=(x_1,x_2)$ belongs to a two-dimensional torus $\mathbb{T}^2$. That is, we impose periodic boundary conditions in space. We let constants freely depend on \(\mathcal{Z}_0, \nu, \max_k |q_k| + |q_k|^{-1}\). 

The pair $(u_t,\mathbf{x}_t)$, which comprises the velocity and position information of fluid particles at time $t$, is referred to as the \emph{Lagrangian process}. 
For systems \eqref{Lagrange1}-\eqref{2NS}, the chaoticity of the Lagrangian flow $\mathbf{x}_t$ (or \emph{Lagrangian chaos}) generally corresponds to the property that its top Lyapunov exponent is strictly positive.

In this paper, we consider mildly degenerate noise acting only on low modes \(\mathcal{Z}_0=\{k\in\mathbb Z^2:0<|k|\le N_*\}\), where \(N_*=N_*(\mathcal{E}_0,\nu)\) is chosen large enough to include all unstable directions, $\mathcal{E}_0:=\sum_{k\in \mathcal{Z}_0}q_k^2$. We prove that the Lagrangian flow generated by \eqref{Lagrange1}-\eqref{2NS} has a strictly positive top Lyapunov exponent. More precisely, we establish:
\begin{theorem} \label{L1}
For system \eqref{Lagrange1}-\eqref{2NS}, there exist an $N_*:=N_*(\mathcal{E}_0,\nu)$ and a deterministic constant $\lambda_+>0$ such that the following limit holds:
\begin{equation} \label{lambda+}
	\lim\limits_{t \to \infty}\frac{1}{t}\mathop{\rm log}|D_x\mathbf{x}_t|=\lambda_+  \,\,\,\, \text{for} \,\,\, \,\mu^1\times \mathbb{P}-a.e.\, (u_0,x,\omega),
\end{equation}  
where $D_x\mathbf{x}_t$ refers to the Jacobian matrix of $\mathbf{x}_t:\mathbb{T}^2\to \mathbb{T}^2$ taken at $x$ and $|D_x\mathbf{x}_t|$ denotes its norm, $\mu^1$ denotes the  unique stationary measure of Lagrangian process $(u_t,\mathbf{x}_t)$. Here, the positivity of $\lambda_+$ indicates exponential sensitivity of Lagrangian trajectories with respect to the initial data.
\end{theorem}
\begin{remark} \label{1.2}
In fact, we can present the following stronger result, which helps support findings on scalar turbulence (c.f. \cite{JEMS,BBP22}): 
\begin{equation} \label{lambda+-P}
	\lim\limits_{t \to \infty}\frac{1}{t}\mathop{\rm log}|D_x\mathbf{x}_t\nu|=\lambda_+>0 \,\, \text{for} \,\,\, \mu^P\times \mathbb{P}-a.e.\, (u_0,x,v,\omega),
\end{equation}
where $\mu^P$ denotes the unique stationary measure for projective process (c.f. Definition \ref{Lagrange-process-def}). 
\end{remark} 	

\subsection{Background and motivations} \label{Relevant literature}

The stochastic 2D incompressible Navier–Stokes equations on a periodic domain serve as an approximation for three-dimensional turbulence in thin domains or as a simplified model of the 3D Navier-Stokes equations. Accordingly, this system is a standard model for the statistical study of randomly forced turbulence (see reference \cite{turbulence} and the references therein). To characterize the long-time behavior of fluids under random perturbations, the theory of ergodicity and mixing for stochastic Navier–Stokes equations has been relatively well developed (see \cite{JFA-12,JFA-41,JFA-42,JFA-45,JFA-69,JFA-87,HM06,HM11a} and references therein). However, ergodicity and mixing represent only the starting point of research in statistical fluid mechanics. To gain a deeper understanding of the nature of turbulence, further studies from other perspectives are still required.


In recent years, some progress has been made in studying fluid turbulence from the perspective of chaos. Related research has primarily developed along two lines: one based on the velocity field, known as Eulerian chaos, and the other based on particle trajectories, referred to as Lagrangian chaos. Both concepts themselves, as well as the relationship between them, have been extensively discussed in the physics literature (see \cite{AGM96, V2-BJPV98, CFVP91, FdCN01, GV94}). Regarding Eulerian chaos, Bedrossian et al. investigated Galerkin approximations of the 2D Navier–Stokes system driven by degenerate noise, establishing Eulerian chaos and providing quantitative lower bounds \cite{BBPS22d, BPS24}. However, the Eulerian chaos of the full stochastic Navier–Stokes system remains an open problem. In terms of Lagrangian chaos, Bedrossian et al. \cite{JEMS} gave the first rigorous proof of Lagrangian chaos induced by a continuous-time stochastic fluid model. Specifically, they established the result for the 2D Navier–Stokes equations driven by time-white noise with non-degenerate excitation acting on the high-frequency modes. Building on this noise setup and the established results, their subsequent work derived a series of important conclusions, including almost sure exponential mixing of passive scalars \cite{BBPS22a}, enhanced dissipation \cite{BBPS21}, and a rigorous version of Batchelor law \cite{BBPS22b}. Following \cite{JEMS}, Cooperman et al. \cite{NS} obtained Lagrangian chaos for the 2D incompressible Navier–Stokes equations driven by highly degenerate noise. Recently, \cite{V2} extended this framework to the case of highly degenerate non-Gaussian bounded noise, though conclusions are currently mainly obtained at the discrete-time level. 

Overall, the 2D stochastic Navier–Stokes equations have become a fundamental model for studying Lagrangian chaos within a rigorous mathematical framework. At the same time, existing proofs often rely strongly on the noise structure and the corresponding technical apparatus: the argument in \cite{JEMS} depends on the strong Feller property of the Markov semigroup, which typically requires the noise to be non-degenerately excited in high-frequency modes, while the treatment of highly degenerate noise in \cite{NS} requires introducing a rather intricate chain of technical steps, including the verification of H\"{o}rmander-type conditions and Malliavin calculus. These limitations result in a higher barrier to understanding and, when extending the study to more complex fluid models, necessitate repeating lengthy Lie algebraic computations to reconstruct key steps such as the regularization of the Markov semigroup. Consequently, a natural research direction is to seek to establish Lagrangian chaos under a noise structure that reflects the `large-scale stirring mechanism' --where the noise is non-degenerate primarily in low-frequency directions--while remaining sufficiently non-degenerate to allow for a more transferable proof mechanism.

Motivated by these, we take the 2D stochastic incompressible Navier–Stokes equations as the core model and aims to establish a Lagrangian chaos analysis framework that is physically reasonable and technically more transparent. This framework is intended to promote the extension of the relevant theory to more general multi-physics coupled fluid models. The random forcing adopted here is a mildly degenerate noise: it acts only on the first \(N_*(\nu)\) modes, where \(N_*(\nu)\) varies with the viscosity coefficient \(\nu\) and covers all unstable directions of the system. To our knowledge, this is the first work that studies and establishes Lagrangian chaos for a fluid equation under such a mildly degenerate noise setting. This noise structure combines an intuitive physical interpretation aligned with large-scale stirring with a sufficient number of randomly forced directions. It allows the analysis to proceed with significantly reduced reliance on technical devices such as H\"{o}rmander-type Lie algebraic computations, thereby leading to a more direct and transparent line of argument and contributing to the understanding of the chaos generation mechanism.

\subsection{Setup and notations} \label{notaion}
Let $\mathbb{T}^2={[0,2\pi]}^2$ denote the period box. We define the following real-valued Fourier basis on this space.
\begin{equation}\label{ek}
e_k(x) := 
\begin{cases} 
	\sin(k \cdot x), & k\in \mathbb{Z}_+^2, \\
	\cos(k \cdot x), & k\in \mathbb{Z}_-^2,
\end{cases}
\end{equation}
where $\mathbb{Z}_+^2:=\{(k_1,k_2)\in \mathbb{Z}^2:k_2>0\}\bigcup \{(k_1,k_2)\in \mathbb{Z}^2:k_1>0,k_2=0\}$ and $\mathbb{Z}_-^2=-\mathbb{Z}_+^2$. For convenience, we rewrite 
$$
\sum_{k \in \mathcal{Z}_0} q_k \nabla^{\perp} \Delta^{-1} e_k(x) \, W^k_t
$$
as 
$$
\sum_{k \in \mathcal{Z}_0} q_k\gamma_k e_k(x) \, W^k_t,
$$
where $\gamma_k:=-k^{\perp}/{|k|^2}$. Under our noise assumption, $q_k\neq 0$ for any $k\in\mathcal{Z}_0$. When a colored forcing convention is used, one may take 
\begin{align} \label{qk}
q_k\approx |k|^{-\alpha} \;\;\;\;\text{for} \;\;0<|k|<N_*,
\end{align}
where $\alpha\in (5,\infty)$ is as in the classical coloring assumption (c.f. \cite{JEMS}).

Next, we first present the definition of the base process and recall the well-posedness and ergodicity results required in the subsequent analysis. 
\begin{definition} \label{base-process-def}
The base process refers to the Markov process \(\{u_t\}_{t\geq 0}\) on \(\mathcal{H}\) which solves the incompressible stochastic Navier-Stokes equation \eqref{2NS} with initial data given by \(u_0\). 
\end{definition}

\begin{proposition} [Well-posedness and ergodicity for base process,\cite{HM06}]\label{base-posedness-ergodicity}
For the system $\eqref{2NS}$, the following holds for any $T>0$:
\begin{enumerate}[label=(\alph*), leftmargin=3em]
	\item For all functions \( u \in \mathcal{H} \) and with probability 1, there exists a unique mild solution \( u_t \in C\big([0, T]; \mathcal{H}\big) \) with \( u_0 = u \). As a function of the noise sample \( \omega \in \Omega \), the solution \( u_t \) is measurable, \( \mathcal{F}_t \)-adapted, and belongs to \( L^p\big(\Omega; C([0, T]; \mathcal{H})\big) \) for all \( p \geq 1 \). Moreover, \( (u_t) \) itself is a Feller Markov process on \( \mathcal{H} \).
	\item The Markov process \( u_t \) admits a unique Borel stationary measure \( \mu \) in \( \mathcal{H} \).
\end{enumerate}
\end{proposition} 

Then, we introduce the Lagrangian process and the projective process studied in this work. Under the assumption that the velocity field $u$ is sufficiently regular, these processes are well-defined and possess the Markov property, which is ensured by verifying the independent-increments hypothesis $\textbf{(H1)}$ in Subsection \ref{RDS framework}.

\begin{definition} \label{Lagrange-process-def}
Given a geodesically complete Riemannian manifold \((M,\hat{g})\) and an associated vector field map \(\Theta: H^{5}_0(\mathbb{T}^2;\mathbb{R}^2) \to \mathcal{X}_\mathrm{loc}^1(M)\) \emph{(}where \(\mathcal{X}_\mathrm{loc}^1({M})\) denotes the set of \(C_{\mathrm{loc}}^{1}\) vector fields on \({M}\)\emph{)}--associating to any velocity field \(u_t\) a vector field \(\Theta_{u_t}\) on \(M\)--we define the associated Markov process \(\{u_t,\mathbf{p}_{t}\}_{t\geq 0}\) where \((u_t)\) is the base process and \(\mathbf{p}_{t}\) solves
\begin{equation} \label{Lagrange-process-eq}
	\dot{\mathbf{p}}_{t} = \Theta_{u_{t}}(\mathbf{p}_{t}),  \quad  \mathbf{p}_{0} \in M.
\end{equation}
In particular, the Lagrangian process \((u_t, \mathbf{x}_t)\) and the projective process \((u_t, \mathbf{x}_t, v_t)\) correspond to the choices \(M = T^2\) and \(M = T^2 \times S^1\), respectively, with  
\begin{align}\label{l2}
\Theta_{u}^1(\mathbf{x}) = u(\mathbf{x}), \quad
\Theta_{u}^P(\mathbf{x},v) = \bigl(u(\mathbf{x}),\, \Pi_{v} Du(\mathbf{x}) v \bigr),
\end{align}
where \(\Pi_v = \text{Id} - v \otimes v\)  is the orthogonal projection from \(\mathbb{R}^2\) onto the tangent space of \(S^1\) \emph{(}viewing $v$ as a unit vector in \(\mathbb{R}^2\)\emph{)}.
\end{definition}

\begin{remark} \label{L-wp}
	Note that the flow \( \mathbf{p}^t \) is a well-defined diffeomorphism since the velocity field \( u_t \) belongs to $\mathcal{H}$ \emph{(}so it is at least \( C^2 \) by Sobolev embedding\emph{)}. This gives rise to an \( \mathcal{F}_t \)-adapted, Feller Markov process \( (u_t,\mathbf{p}_t) \) on \( \mathcal{H}\times M\) defined by \eqref{Lagrange-process-eq}.
\end{remark}
\subsection{Outline of the proof and contributions}

Given that the Lagrangian system associated with the 2D incompressible Navier–Stokes equations is conservative and the forcing is driven by mildly degenerate noise, the analysis in this paper follows the framework of continuous random dynamical systems, the multiplicative ergodic theorem, and the Furstenberg criterion. The specific form of Furstenberg criterion adopted here is taken from \cite{NS}. Within this framework, we establish a criterion for Lagrangian chaos applicable to the case of mildly degenerate noise (c.f. Theorem \ref{FC}): it suffices to verify that the derivative cocycle \(\mathcal{A}\) of system \eqref{Lagrange1} satisfies the integrability condition \textbf{(A$_1$)}, that the Lagrangian process and its projective process satisfy the asymptotic strong Feller property \textbf{(A$_2$)}, and that the approximate controllability condition \textbf{(A$_3$)} holds. Here, \textbf{(A$_1$)} follows from regularity estimates for the base process, while \textbf{(A$_3$)} has been established in \cite[Section 7]{JEMS} by constructing smooth controls via shear and cellular flows, together with a stability analysis and the positivity of Wiener measure. Consequently, the proof of our main result ultimately reduces to establishing \emph{short-time asymptotic gradient estimates} for the Markov transition semigroups of the Lagrangian and projective processes under mildly degenerate noise.(c.f. Theorem \ref{SAGE}).



The key to establishing the asymptotic gradient estimates for the transition semigroup lies in the application of the Malliavin calculus that was employed in \cite{HM06,HM11a}. The main idea is to construct a small perturbation \(g_{_{t}}\) of the driving Wiener path such that, at time \(t\), the variation of the state \((u_t,\mathbf{p}_{_{t}})\) induced by this perturbation cancels as much as possible the variation generated by a small perturbation \(\mathfrak{h}\in \mathcal{H}\times TM\) of the initial condition, thereby achieving an effective matching between the Malliavin derivative and the Fr\'{e}chet derivative. Consequently, the problem reduces to constructing a smooth control together with bounds on its cost, and to estimating the associated error terms (c.f. \eqref{3.9}).
\subsubsection{Discussion of control construction and error estimates}
For the case driven by mildly degenerate noise, a standard strategy for obtaining asymptotic gradient estimates of the transition semigroup for the base process is the low-high frequency splitting method adopted in \cite[Subsection 4.5]{HM06}. One imposes a control on the low-frequency component and uses finite-time stabilization to eliminate the low-frequency error after a finite time. No control is applied to the high-frequency component, and the corresponding error is allowed to decay due to high-mode dissipation. Combined with error inversion for control design, this ensures bounded control cost and allows the total error to satisfy the required estimates. It is important to stress that the viability of this scheme ultimately hinges on the invertibility of the low-frequency diffusion coefficient.


Unlike \cite{HM06}, which treats only the base process, we must also handle the additional degrees of freedom \(\mathbf{x}_t\) and \(v_t\) in the extended system. These variables are not directly driven by the noise. Instead, randomness reaches them through the noise acting on \(u_t\) and is then transmitted via \eqref{Lagrange-process-eq}–\eqref{l2}. As a result, the degeneracy of the extended system is substantially enhanced compared to the system that only considers the velocity field. A direct consequence is the following. If one naturally includes \((\mathbf{x}_t,v_t)\) in the low-frequency subsystem by taking \(\Pi_l(u,\mathbf{x},v)=(u^l,\mathbf{x},v)\), then the associated low-frequency diffusion coefficient \(\hat {Q}_l:=\Pi_l \hat {Q}\) (cf. \eqref{abstract system}) is necessarily non-invertible. Consequently, the error-inversion mechanism used in \cite[Section 4.5]{HM06} can no longer be applied. A seemingly plausible alternative is turn to place the manifold components \((\mathbf{x}_t,v_t)\) into the high-frequency subsystem in order to preserve invertibility of the low-frequency diffusion coefficient. However, \(\mathbf{x}_t\) and \(v_t\) lack the dissipative structure of the high-frequency velocity modes, placing them on the high-frequency side would necessitate imposing additional control on them. This ultimately forces the analysis back into the infinite-dimensional setting, which is precisely one of the difficulties encountered in \cite{NS,HM11a,JFA} for highly degenerate noise. 

For this reason, and exploiting the finite-dimensional nature of the low-frequency subsystem, we use the decomposition \(\Pi_l(u,\mathbf{x},v)=(u^l,\mathbf{x},v)\) and \(\Pi_h(u,\mathbf{x},v)=u^h\). We implement controls only on the low-frequency subsystem and fully exploit high-mode dissipation to obtain the desired estimates. To overcome the obstruction caused by the non-invertibility of \(\hat {Q}_l\), we abandon the traditional error-inversion design and instead introduce a new finite-dimensional Malliavin framework. We construct an \emph{invertible partial (finite-dimensional) Malliavin matrix} for the low-frequency subsystem and use it to build an explicit smooth control. This provides representations and estimates for the Malliavin derivative, together with corresponding representations and bounds for the error terms. Notably, the invertibility result for the partial Malliavin matrix here is stronger than the non-degeneracy condition in \cite{NS,JFA,HM11a}, which holds only on a cone. As a result, the control can be directly constructed without the need to use a regularized version of the Malliavin matrix, and the error estimates are simplified accordingly.

More specifically, for a short time \(T_0(<1)\), we aim to construct a control \(g_{_{t}}=(g^{l}_{_{t}},0)\) over the interval \([0,T_0]\) such that the low-frequency part of the corresponding perturbed Malliavin derivative satisfies, as closely as possible, the condition
\begin{align} \label{1-1}
\Pi_l\big(\mathcal{D}^{g}(u_{_{T_0}},\mathbf{p}_{_{T_0}})\big) =:\mathcal{D}^{g^{l}}(u^l_{_{T_0}},\mathbf{p}_{_{T_0}})=\Pi_l\big(\mathcal{J}_{0,T_0}\mathfrak{h}\big),
\end{align}
where the full Jacobian \( \mathcal{J}_{s,t} \) is the linear map from \( \mathcal{H} \times T_{\mathbf{p}_s}M \) to \( \mathcal{H} \times T_{\mathbf{p}_t}M \) that satisfies \eqref{Jh1}. Note that if there were no coupling terms between the high and low frequencies, then in the ideal scenario, the Duhamels's formula would yield
\begin{align} \label{1-2}
\mathcal{D}^{g^{l}}(u^l_{_{T_0}},\mathbf{p}_{_{T_0}}) = \int_0^{T_0}R_{s,T_0}^l\hat{Q}_lg_{_{s}}^lds,
\end{align}
where the finite-dimensional matrix \( R_{s,t}^l \) is a linear map from \( \mathcal{H}_l \times T_{\mathbf{p}_s}M \) to \( \mathcal{H}_l \times T_{\mathbf{p}_t}M \) that satisfies \eqref{Rl1}. Naturally, our goal is to use a least-squares approach to design a control $g_{_{s}}^l$ such that 
\begin{align} \label{1-3}
\int_0^{T_0}R_{s,T_0}^l\hat{Q}_lg_{_{s}}^lds=\Pi_l\big(\mathcal{J}_{0,T_0}\mathfrak{h}\big).
\end{align}
To address the fact that \(R_{s,t}^l\) is not adapted, we introduce its (adapted) inverse matrix \(S_{s,t}^l\). After a time \(\tau_0 \in [1/2,1)\), we then define the finite-dimensional \emph{partial Malliavin matrix}
\[
\mathcal N_{T_0}^l=\int_{\tau_0}^{T_0} S_{\tau_0,s}^l\hat Q^l\big(S_{\tau_0,s}^l\hat Q^l\big)^\top\,ds .
\]
The reason for setting the lower limit of integration to \(\tau_0\) is to fully exploit the dissipation effect of the high-frequency modes in the system, ensuring that the high-frequency part of the error satisfies the required estimates. A detailed explanation of this point will be provided at the end of this subsection. Accordingly, set \(g_t^l=0\) on \([0,\tau_0]\) and activate the control only on \([\tau_0,T_0]\). More specifically, once the invertibility of \(\mathcal N_{T_0}^l\) is established, we can construct on \([\tau_0,T_0]\) the following smooth low-frequency control \(g^l_{_{t}}\): 
\[
g^l_{_{t}}=(S_{_{\tau_0,t}}^l\hat{Q}_l)^\top(\mathcal{N}_{_{T_0}}^l)^{-1}S_{_{\tau_0,T_0}}^l\Pi_l(\mathcal{J}_{0,T_0}\mathfrak{h})
\]  
which yields a global control \(g_t^l\) defined on \([0,T_0]\) and ensures that the required matching relation \eqref{1-3} holds. To maintain the flow of the exposition, the proof of the invertibility of \(\mathcal N_{T_0}^l\) is postponed to the next subsection. Finally, the cost of this non-adapted control is estimated with the \(L^2\)-isometry for the Skorohod integral.

We now turn to the error estimate, beginning with the low-frequency component. In this work we adopt a frequency-splitting framework and build the Malliavin structure only on a finite-dimensional low-frequency subsystem. Compared with the approach in \cite{NS,JFA,HM11a}, which is based on the full infinite-dimensional Malliavin matrix, this setting has two consequences.

On the one hand, the low-mode equation now inevitably contains coupling terms between low and high modes which are similar to those in \cite{JEMS}. These terms involve Malliavin derivatives of the high-mode components, and hence form the main additional source of error in the low-mode estimate. The difference is that the control constructed earlier in this paper does not depend on Malliavin derivatives. This significantly simplifies the subsequent analysis. Indeed, unlike in \cite{JEMS}, we do not need to first introduce auxiliary processes $(\zeta_t,\xi_t)$ depending on Malliavin derivatives, nor do we need to truncate $u$ in order to prove the existence and uniqueness of these auxiliary processes. Instead, we can directly establish the well-posedness of the Malliavin derivative equations for both the low and high modes, and then obtain the corresponding derivative estimates. In this way, we complete the estimate for this part. On the other hand, since the low-mode Malliavin matrix $\mathcal N_{T_0}^l$ is invertible, the matching relation \eqref{1-3} can be realized exactly. Therefore, we avoid the error terms arising from the least-squares construction in \cite{NS}, and the corresponding computations are also greatly simplified.

For the high-frequency component of the error, we exploit the dissipative structure of the high-frequency modes and decompose it into two terms, namely the purely high-frequency dissipative term and the high–low frequency coupling term. The key point is that the purely dissipative term can be controlled directly by the decay of the high-frequency dynamics, yielding a bound consistent with the target error estimate of the form \eqref{erro-estimate1}. To make this dissipation effective, the low-mode control must not act from time \(0\). Otherwise, when estimating the purely dissipative term one inevitably obtains a factor of the form \((1+q(T_0))e^{\eta V(u_0)}\), where \(q(T_0)\to0\) as \(T_0\to0\), which is incompatible with \eqref{erro-estimate1}. To circumvent this difficulty, \cite[Section 6]{JEMS} introduced an auxiliary process \(z_t\) and a truncated process \(w_t^\rho\). In contrast, we avoid the additional auxiliary and truncation procedures by adopting a more direct time-splitting strategy. We first let the system evolve freely on \([0,\tau_0]\), that is, we set the control to zero, so that the high-frequency component decays sufficiently fast due to dissipation. We then activate the low-frequency control only on \([\tau_0,T_0]\). Consequently, the relevant key term can be compressed into
\[
\Big(e^{C\tau_0}N_*^{-2}+\frac{e^{C\widetilde{T}_0}}{4\nu N_*^2}\widetilde{T}_0^{3/2}+\frac{1}{4\nu N_*^2} \, e^{C\widetilde{T}_0}\widetilde{T}_0^{5/2}+\widetilde{T}_0^{3}\Big)e^{\eta V(u_0)},\qquad \widetilde T_0:=T_0-\tau_0,
\]
which meets the requirement of \eqref{erro-estimate1}.

\subsubsection{Discussion of the nondegeneracy of the partial Malliavin matrix}

We now address the non-degeneracy of the partial Malliavin matrix \(\mathcal N_{T_0}^l\). This property plays a central role in both the construction of the control and the ensuing error estimates. Let us emphasize that \(\mathcal N_{T_0}^l\) is defined on the finite-dimensional space \(\mathcal {H}_l\times TM\) and is a finite-dimensional adapted matrix. In contrast to the highly degenerate setting--where one must deal with the full Malliavin matrix $\mathcal{M}$--we can obtain a stronger conclusion here: \(\mathcal N_{T_0}^l\) is invertible. Moreover, the non-degeneracy  bound in Proposition \ref{M-1} essentially reduces to establishing the probabilistic spectral bound \eqref{spectral bound1} for the Malliavin matrix \(\mathcal N_{T_0}^l\).

A standard approach to such a probabilistic spectral bound consists of two steps. The first is to establish a H\"{o}rmander-type spanning condition, which requires that the Lie brackets of H\"{o}rmander type cover all unstable directions of the system. More precisely, one proves the following quantitative lower bound only for those vectors $h$ that lie in a certain cone: for every \( N > 0 \), there is a \( n > 0 \) and a finite set \( \mathfrak{B} \subset \mathcal{B}_n \)  such that for each \(\alpha>0\), we define the cone
\[
\mathcal K_{N,\alpha}
:=\Bigl\{h\in\mathcal H:\ \|P_Nh\|^2>\alpha\|h\|^2\Bigr\},
\]
where \(\{P_N\}\) is a sequence of projection operators onto successively larger spaces. On this cone, the following lower bound holds:
\begin{align}\label{1.12}
\sum_{b(u)\in\mathfrak{B}}\langle h,\,b(u)\rangle^2\ge \alpha C\|h\|^2, \quad \forall\,h\in\mathcal \mathcal{K}_{N,\alpha},
\end{align}
where \( C > 0 \) is a constant independent of \( \alpha \), $\mathcal{B}_0:=\text{span}\{e_k,\, k\in \mathcal{Z}_0\}$, $\mathcal{B}_n:=\text{span}\{[E,F],[E,e_k],\,E:E\in \mathcal{B}_{n-1},\, k\in \mathcal{Z}_0\}$ and $F$ is the drift term.
The second step is to differentiate \(\langle \mathcal{J}_{s,t}e_k,h\rangle\) in time and derive, on a high-probability event, an `implication estimate': if \(\langle \mathcal{M}_T h,h\rangle\) is sufficiently small, then it necessarily forces the projections onto the directions in $\mathfrak{B}$ to become small simultaneously, i.e.
\begin{align}\label{1.13}
\langle \mathcal{M}_T h,h\rangle<\varepsilon\|h\|^2 \ \Rightarrow\ \sum_{b(u)\in \mathfrak{B}}\langle h,\,b(u)\rangle^2 \le  \varepsilon^q\|h\|^2.
\end{align}
Combining this upper implication with the H\"{o}rmander lower bound \eqref{1.12}, and choosing \(\varepsilon\) appropriately, one obtains the probabilistic spectral bound.


Unlike the classical setting where only the velocity field is involved, our extended system additionally includes degrees of freedom \((\mathbf{x}_t, v_t)\) on the manifold. Consequently, the non-degeneracy argument must also incorporate the vector fields on the tangent bundle \(TM\). To establish the analogue \eqref{1.14} of the lower bound \eqref{1.12} required in the first step, we need to exploit the Lie brackets \([\overline{F}, e_k\gamma_k]\) to generate new directions in both the position and the projection components, where \(\overline{F}\) is the manifold vector field defined in \eqref{F1}. This leads to the following spanning condition on the tangent bundle:
\[
\operatorname{span}\{[\overline{F},e_k\gamma_k](\mathbf{x},v):k\in\mathcal Z_0\}=T_{\mathbf{x}}\mathbb T^2\times T_vS^1.
\]
Building on this, the mildly degenerate noise structure considered in this work allows the above first step to be completed along a substantially simpler route than in the highly degenerate case treated in \cite{NS}. More specifically, in \cite{NS} one typically has to perform repeated time differentiations in order to successively generate new Lie bracket directions on \(P_N\mathcal H\times TM\) for sufficiently large \(N>0\). In contrast, we work with the finite-dimensional partial Malliavin matrix associated with the low-frequency subsystem, so that the H\"{o}rmander condition only needs to be verified on \(\mathcal H_l\times TM\). On the one hand, the spanning property on \(\mathcal H_l\) follows directly from the structure of the mildly degenerate noise. On the other hand, the verification on \(TM\) relies on an explicit and tractable representation of the low-frequency object \(S_{\tau_0,t}^l e_k\gamma_k\). More precisely, since \(S_{\tau_0,t}^l\) is adapted, we may apply It\^{o}'s formula to \(S_{\tau_0,t}^l e_k\gamma_k\) and expand it explicitly. In this expansion, terms involving the bracket \([\overline{F},e_k\gamma_k]\) appear naturally. Combining this with the above spanning condition on \(TM\), we obtain the following spectral-type lower bound:
\begin{align} \label{1.14}
\max\Big\{\big|\langle e_k\gamma_k,\mathfrak{h}^l\rangle\big|,\ \big|\langle \Upsilon_L e_k\gamma_k,\mathfrak{h}^l\rangle\big|:k\in\mathcal{Z}_0\Big\}
\ge \frac{C(N_*,\nu)}{1+\|u^l\|_{\mathcal{H}_l}} \|\mathfrak{h}^l\|_l.
\end{align}
where \(\Upsilon_L e_k\gamma_k\) is defined in \eqref{Upsilon}.

In the second step, we need to establish the counterpart, in our present setting, of the general implication \eqref{1.13}, namely \eqref{upper}. This is obtained by differentiating in time the quantity \(\langle S_{\tau_0,t}^l e_k\gamma_k,\mathfrak{h}^l\rangle\). We then combine \eqref{upper} with the spectral-type lower bound to deduce the desired probabilistic spectral bound.

Finally, we emphasize that this work develops a unified analytical framework: on the low-frequency component, a control mechanism is used to capture the effective action of the stochastic forcing; on a finite-dimensional level, Malliavin calculus is employed to extract the required non-degeneracy; and high-frequency dissipation is leveraged to suppress small-scale instabilities. This framework is broadly applicable to the systematic study of Lagrangian chaos in incompressible fluid systems driven by mildly degenerate noise. In subsequent work, we will formulate more explicit abstract criteria tailored to this class of mildly degenerate noise and apply them to the three-dimensional hyperviscous Navier–Stokes equations as well as to multiphysics coupled fluid models, such as the two-dimensional Boussinesq and magnetohydrodynamics equations. Our aim is to further deepen the understanding of the mechanisms responsible for the emergence of chaos in statistical fluid mechanics.

It is worth noting that, in our recent work on the 2D Boussinesq equation \cite{CZ25}, the noise acts only on the temperature component and is highly degenerate; consequently, first-order Lie brackets \([\hat {F}, e_k\gamma_k]\) do not generate new directions on the tangent bundle, and higher-order brackets become necessary. Although the underlying principle remains the same--namely, achieving controllability by effectively `projecting' infinite-dimensional directions onto tangent-bundle vector fields--the Lie-bracket computations required both to verify the manifold spanning condition and to derive the relevant upper bounds are technically cumbersome. By contrast, the present framework preserves the same core mechanism while substantially reducing the computational overhead, thereby providing a more streamlined and practical route for extensions to the above classes of equations and beyond.

The paper is organized as follows. Section \ref{SEC2} reviews several basic results from the theory of random dynamical systems and incorporates the refined Furstenberg's criterion from \cite{NS} into a positivity criterion for the top Lyapunov exponent of Lagrangian systems associated with incompressible fluid models driven by mildly degenerate noise. The problem is then reduced to establishing short-time asymptotic gradient estimates for the Lagrangian process and projective process. In Section \ref{SEC3}, these gradient estimates for the extended system are reformulated in terms of the construction of controls for the low-frequency subsystem, together with corresponding error estimates, and the main result is proved under the assumption that a partial Malliavin matrix is invertible and satisfies suitable non-degeneracy bounds. Section \ref{SEC4} establishes the required non-degeneracy estimates for the partial Malliavin matrix and proves its invertibility. Finally, the Appendix collects proofs of several auxiliary results used in the paper.

\section{Fundamental mathematical framework and Furstenberg criterion}\label{SEC2}

Our objective is to prove that the top Lyapunov exponent is almost everywhere a positive constant. The proof proceeds in two stages: first, within the framework of random dynamical systems (RDS), we employ the Multiplicative ergodic theorem (MET) to establish the almost-everywhere existence and constancy of the top Lyapunov exponent; second, we apply the Furstenberg's criterion to verify its positivity.
 
To this end, this section first introduces the necessary notions and results concerning random dynamical systems and the Multiplicative ergodic theorem, then states the classical Furstenberg's criterion, and finally provides a criterion for the positivity of the top Lyapunov exponent adapted to the noise assumptions of this paper. It should be noted that this part essentially reformulates and synthesizes results from \cite{NS}, without introducing fundamentally new ideas.
\subsection{RDS framework and Multiplicative ergodic theorem} \label{RDS framework}
Let \((Z, \mathcal{B})\) be a measurable space, and \((\Omega, \mathcal{F}, \mathbb{P}, (\theta^t)_{t \in [0,\infty)})\) be a metric dynamical system with time index set \([0,\infty)\). The mapping  
\[
\mathcal{T} : [0,\infty) \times \Omega \times Z \to Z, \quad (t, \omega, z) \mapsto \mathcal{T}_\omega^tz
\]  
denotes a \emph{measurable random dynamical system} covering this metric system. If for every \(\omega \in \Omega\), the mapping \(\mathcal{T}_\omega: [0, \infty) \times Z \to Z\) belongs to the space \(C_{u,b}([0, \infty) \times Z, Z)\), then the random dynamical system is called \emph{continuous RDS}. Here, for metric spaces \(V, W\), the notation \(C_{u,b}(V, W)\) denotes the set of all continuous maps \(E: V \to W\) such that for every bounded set \(O \subseteq V\), the restriction \(E|_O\) is uniformly continuous and the image \(E(O)\) is a bounded subset of \(W\).

To ensure that the process \( z_t \) as above is Markovian, it is often further required that the continuous random dynamical system \( \mathcal{T} \) satisfies the usual \emph{independent increments assumption}.\\
{\textbf{(H)}} For all \( s, t > 0 \), we have that \(\mathcal{T}_{\omega}^{t}\) is independent of \(\mathcal{T}_{\theta^t \omega}^s\). That is, the \(\sigma\)-subalgebra \(\sigma(\mathcal{T}_{\omega}^{t}) \subset \mathcal{F}\) generated by the \(C_{u,b}(Z, Z)\)-valued random variable \(\omega \mapsto \mathcal{T}_{\omega}^{t}\) is independent of the \(\sigma\)-subalgebra \(\sigma(\mathcal{T}_{\theta^t \omega}^s)\) generated by \(\omega \mapsto \mathcal{T}_{\theta^t \omega}^s\).

As can be seen from the regularity results for solutions to the Navier-Stokes equations in \cite[Appendix A]{JEMS}, we have the following fundamental result.
\begin{proposition}[RDS for base, Lagrangian and projective processes] \label{RDSf}
Let  
\( \mathcal{U} : [0, \infty) \times \Omega \times \mathcal{H} \to \mathcal{H}, (t, \omega, u) \mapsto \mathcal{U}^t_\omega(u) \) 
denote the mapping sending, for a given \( t \geq 0 \) and \( \mathbb{P} \)-generic \( \omega \in \Omega \), a given \( u \in \mathcal{H} \) to the time-\( t \) vector field \( u_t \) conditioned on \( u_0 = u \). Then \( \mathcal{U} \) is a continuous RDS on the space \( Z = \mathcal{H} \) satisfying condition \textbf{\emph{(H)}}. 

Similarly, the random ODE \eqref{l2} defining the auxiliary process  
\( \mathbf{x}_t = \mathbf{x}^t_{\omega, u_0} \mathbf{x}_0 \) and $v_t = v^t_{\omega, u_0,\mathbf{x}_0}v_0$ 
are well posed, and we conclude as before that the corresponding mapping  
\( \mathfrak{L} : [0, \infty) \times \Omega \times \mathcal{H} \times T^2 \to \mathcal{H} \times T^2 \) and \(\mathfrak{P} : [0, \infty) \times \Omega \times \mathcal{H} \times T^2 \times S^1 \to \mathcal{H} \times T^2 \times S^1\)
for the Lagrangian process \( (u_t, \mathbf{x}_t) \) and projective process \( (u_t, \mathbf{x}_t,v_t) \) are continuous RDS satisfying \textbf{\emph{(H)}} on the space \( Z = \mathcal{H} \times T^2 \) and \( Z = \mathcal{H} \times T^2 \times S^1\) respectively. 
\end{proposition}

Now, the process $(u_t,\mathbf{p}_t)$ is Markovian. We write \( P_t : \mathcal{H} \times M \to \mathcal{P}(\mathcal{H} \times M) \) to denote the time \( t \) transition kernel, where \( \mathcal{P}(X) \) denotes the space of probability measures on \( X \). We also let \( P_t \) act adjointly on observables \( \varphi : \mathcal{H} \times M \to \mathbb{R} \) by pulling back:
\[
P_t \varphi(u_0,\mathbf{p}_0) = \int \varphi(u_t,\mathbf{p}_t) \, P_t(u_0,\mathbf{p}_0; \mathrm{d}u_t,\mathrm{d}\mathbf{p}_t).
\]

Next, we introduce linear cocycles over random dynamical systems and the classical Multiplicative ergodic theorem.
\begin{definition} \label{cocycle-def}
Let \(\mathcal{T}\) be a continuous RDS as above, and let \((\tau^t)\) be its associated skew product. A \(d\)-dimensional linear cocycle \(\mathcal{A}\) over the base RDS \(\mathcal{T}\) is a mapping \(\mathcal{A}:\Omega \to C_{u,b}([0,\infty) \times Z,\, M_{d \times d}(\mathbb{R}))\) with the following properties:
\begin{itemize}
	\item[(i)] The evaluation mapping \(\Omega \times [0,\infty) \times Z \to M_{d \times d}(\mathbb{R})\) sending \((\omega,t,z) \mapsto \mathcal{A}_{w,z}^t\) is \(\mathcal{F} \otimes \text{Bor}([0,\infty)) \otimes \text{Bor}(Z)\)-measurable.
	\item[(ii)] The mapping \(\mathcal{A}\) satisfies the cocycle property: for any \(z \in Z, w \in \Omega\) we have \(\mathcal{A}_{w,z}^0 = \mathrm{Id}_{\mathbb{R}^d}\), the \(d \times d\) identity matrix, and for \(s,t \geq 0\) we have
	\[
	\mathcal{A}_{w,z}^{s+t} = \mathcal{A}_{\tau^t(w,z)}^s \circ \mathcal{A}_{w,z}^t.
	\]
\end{itemize}
\end{definition}

\begin{remark} \label{tangent-process-cocycle}
The main focus of this paper is the derivative cocycle $\mathcal{A}$ of Lagrangian particle trajectories \( x_t \), where
\begin{align} \label{LC}
\mathcal{A}: [0, \infty) \times \Omega \times \mathcal{H} \times \mathbb{T}^2 \rightarrow M_{2 \times 2}(\mathbb{R}),\;\;\mathcal{A}_{\omega, u, \mathbf{x}}^t = D_x \mathbf{x}_{\scriptstyle{\omega, u}}^t\,,
\end{align}
which corresponds to taking \( \mathcal{T} = \mathfrak{L} \) and \( Z = \mathcal{H} \times \mathbb{T}^2 \).
\end{remark}

The following is a version of the Multiplicative ergodic theorem in \cite{V2-Ose68}.
\begin{theorem}\label{MET}
Let \(\mathcal{T}\) be a continuous RDS as above satisfying condition \textbf{\emph{(H)}}. Let \(\mu \in \mathcal{P}(Z)\) be an ergodic stationary measure for \(\mathcal{T}\), and let \(\mathcal{A}\) be a linear cocycle over \(\mathcal{T}\) satisfying the following integrability condition: 
\begin{equation} \label{integrability}
\mathbb{E} \int_{Z} \left(  \log^+ |\mathcal{A}_{w,z}^t|+ \log^+ |(\mathcal{A}_{w,z}^t)^{-1}|\right) d\mu(z) < \infty,
\end{equation}
where \(\log^+(x) := \max\{0, \log(x)\}\) for \(x > 0\) and \(\mathbb{E}\) is the expectation with respect to \(\mathbb{P}\). Then there are \(r \in \{1, \ldots, d\}\) deterministic real numbers
\[
\lambda_r < \cdots < \lambda_1
\]
and for \(\mathbb{P} \times \mu\)-a.e. \((w, z)\), a flag of subspaces
\[
\{0\} =: F_{r+1} \subsetneq F_r(w, z) \subsetneq \cdots \subsetneq F_2(w, z) \subsetneq F_1 := \mathbb{R}^d
\]
such that
\[
\lambda_i = \lim_{t \to \infty} \frac{1}{t} \log |\mathcal{A}_{w,z}^t v|, \quad v \in F_i(w, z) \setminus F_{i+1}(w, z).
\]
Moreover, for any \(i \in \{1, \ldots, r\}\), the mapping \((w, z) \mapsto F_i(w, z)\) is measurable and \(\dim F_i(w, z)\) is constant for \(\mathbb{P} \times \mu\)-a.e. \((w, z)\).
\end{theorem}

Here, the numbers \(\lambda_i\) are called Lyapunov exponents and $m_i := \dim F_i(\omega, z) - \dim F_{i+1}(\omega, z)$ are their multiplicities. In particular, \(\lambda_1\) is precisely the \emph{top Lyapunov exponent} of interest in this paper.

If the Lagrangian process admits a unique stationary measure \(\mu^1\) (in which case \(\mathbb{P} \times \mu^1\) is an ergodic invariant measure of the random dynamical system \(\mathfrak{L}\)), then under the assumed integrability condition, the Multiplicative ergodic theorem implies that the top Lyapunov exponent of this linear cocycle exists and is almost surely constant with respect to \(\mathbb{P} \times \mu^1\). This completes the first part of the proof.
\subsection{Furstenberg criterion and results}\label{s2.2}
We review the Furstenberg's criterion to decide positivity of the top Lyapunov exponent for conservative systems, and then provide conditions specialized to the noise assumptions of this work.
We first present a standard formulation of the Furstenberg's criterion. This formulation is directly adapted from \cite{Led86} and is further developed in subsequent work published in \cite{JEMS}.
\begin{theorem} \label{F}
If \(\lambda_1 = \lambda_r\), then for each \(z \in Z\) there is a Borel measure \(\nu_z\) on \(S^{1}\) such that
\begin{enumerate}
    \item[(i)] the assignment \(z \mapsto \nu_z\) is measurable, and
    \item[(ii)] for each \(t \in [0, \infty)\) and \((\mathbb{P} \times \mu)\)-almost all \((w, z) \in \Omega \times Z\) (perhaps depending on \(t\)), we have that
    \begin{equation}\label{invariant-measure}
    (\mathcal{A}_{w,z}^t)_*\nu_z = \nu_{\mathcal{T}_{w,z}^t}.
    \end{equation}
\end{enumerate}
\end{theorem}
 
The classical Furstenberg criterion asserts that the top Lyapunov exponent is positive once one rules out every projective family of measures \(\{\nu_z\}_{z\in\operatorname{supp}\mu}\) satisfying the invariance relation \eqref{invariant-measure}. In general, however, measurability of such a family alone does not suffice to exclude these invariant structures. To obtain an effective exclusion, existing arguments typically impose stronger regularity to ensure a weakly continuous selection of \(\{\nu_z\}\): \cite[Section 4]{Bou88} and \cite[Proposition 2.10]{BCZG23} assume that the stationary measure \(\mu\) is mixing in total variation, while related work in \cite{JEMS} requires the Lagrangian process and its projective process to be strong Feller and the associated Markov semigroup to be \(C^0\)-continuous.  

For Navier–Stokes Lagrangian trajectories driven by mildly degenerate noise, neither set of assumptions is available: mixing holds only with respect to the weaker dual Lipschitz metric (hence not in total variation), and the strong Feller property typically fails as well. We therefore adopt the refined Furstenberg's criterion of \cite[Proposition 3.4]{NS} and derive below a criterion tailored to the noise setting of this work.

To formulate sufficient conditions ensuring the positivity of the top Lyapunov exponent for the linear cocycle \eqref{LC} over the random dynamical system $\mathfrak{L}$, we impose the following three assumptions.
\begin{enumerate}
    \setcounter{enumi}{0} 
    \item[\textbf{(A$_1$)}] \textbf{Integrability.} The linear cocycle \(\mathcal{A}\) satisfies 
$$
\mathbb{E} \int_{Z} \left(  \log^+ |\mathcal{A}_{w,z}^t|+ \log^+ |(\mathcal{A}_{w,z}^t)^{-1}|\right) d\mu(z) < \infty.
$$
    \item[\textbf{(A$_2$)}] \textbf{Asymptotic strong Feller property of the extended system.} For system \eqref{2NS} and \eqref{Lagrange-process-eq}-\eqref{l2}, there exists an $N_*=N_*(\mathcal{E}_0,\nu)$ such that if $\mathcal{Z}_0=\{k\in \mathbb{Z}^2,\, 0<|k|\le N_*\}$, then the Markov semigroups associated with the Lagerange process $(u_t,x_t)$ and projective process $(u_t,x_t,v_t)$ are all \emph{asymptotic strong Feller} in $\mathcal{H}\times \mathbb{T}^2$ and $\mathcal{H}\times \mathbb{T}^2\times S^1$ respectively.
    \item[\textbf{(A$_3$)}] \textbf{Approximate controllability of the extended system.} There exist \( z, z' \in \operatorname{supp} \mu \) such that \( z' \) belongs to the support of the measure \( P_{t_0}(z, \cdot) \) for some \( t_0 > 0 \), where $\mu$ is the unique stationary measure for base process. Moreover, there exists $t\ge t_0$ such that 
\begin{enumerate}
\item[(a)] For any \(\mathbf{x}\in \mathbb{T}^2,\,\varepsilon, M > 0\),
\[
\mathbb{P}\left((u_t,\mathbf{x}_t,A_t)\in B_\varepsilon(z') \times B_\varepsilon(\mathbf{x}')\times \{A \in \mathrm{SL}_2(\mathbb{R}) : |A| > M\}|(u_0,\mathbf{x}_0,A_0)=(z,\mathbf{x},\mathrm{Id}) \right)> 0.
\]
\item[(b)] For any \( (\mathbf{x},v) \in \mathbb{T}^2\times S^{1} \), open set \( V \subset S^{1} \), and \(\varepsilon > 0\),
\[
\mathbb{P}\left((u_t,\mathbf{x}_t,v_t)\in B_\varepsilon(z') \times B_\varepsilon(\mathbf{x}') \times V |(u_0,\mathbf{x}_0,v_0)=(z,\mathbf{x},v)\right)> 0.
\]
\end{enumerate}
\end{enumerate}
\begin{theorem}\label{FC}
Suppose that Assumptions \emph{\textbf{(A$_1$)}-\textbf{(A$_3$)}} are satisfied. Then there exists a deterministic constant $\lambda_+ > 0$ such that
\begin{equation}\label{lambda>0}
\lambda_1=\lambda_+ = \lim_{t \to \infty} \frac{1}{t} \log |\mathcal{A}^t_{_{u_0,x}}|>0 \,\,\,\, \text{for} \,\,\, \,\mu^1\times \mathbb{P}-a.e.\, (u_0,x,\omega),
\end{equation}
where $\mu^1$ is the unique stationary measure for the Lagrangian process.
\end{theorem}

\begin{remark}
By the asymptotic strong Feller theory of Mattingly–Hairer and collaborators, assumptions \emph{\textbf{(A$_2$)}}–\emph{\textbf{(A$_3$)}} already imply the \emph{uniqueness} of the stationary measure for the extended system. On the other hand, the compactness of the manifold $M$ together with a super-Lyapunov condition (c.f. Corollary \ref{SLP}) yields the \emph{existence} of an invariant measure. See \cite[Corollary 2.1]{NS} for details. 

Moreover, we also observe that the Lagrangian system induced by the incompressible stochastic Navier–Stokes equation is conservative. Hence
\(
\lambda_{\scriptscriptstyle\sum}:=\sum_i \lambda_i=0,
\)
and therefore
\[
d\,\lambda_1\ge \lambda_{\scriptscriptstyle\sum}=0 \quad \Rightarrow \quad \lambda_1\ge 0.
\]
Thus, to obtain positivity of the top Lyapunov exponent, it suffices to rule out \(\lambda_1=0\), equivalently the spectral degeneracy \(\lambda_1=\lambda_r\). This is precisely the obstruction addressed by the Furstenberg criterion. We verify an approximate controllability property for the extended system, which excludes the continuous invariant structures in the refined Furstenberg's criterion of \cite[Theorems 3.5-3.6]{NS}; the required continuity of these invariant structures is ensured by the asymptotic strong Feller property of the extended system. For further details, see \cite[Section 3]{NS} and \cite[Section 4]{JEMS}. This completes the second part of the proof.
\end{remark}

\begin{proof}[Proof of the Remark \ref{1.2}]
    By utilizing the existence of a unique stationary measure for the projective process (c.f. Assumption \textbf{(A$_2$)}) and the random multiplicative Ergodic Theorem \cite[Theorem~III.1.2]{PE-50}, we can complete the proof.
\end{proof}

Concerning the verification of assumptions \textbf{(A$_1$)}–{\textbf{(A$_3$)}}: assumption {\textbf{(A$_1$)}} follows directly from regularity estimates for the solution; see \cite[Appendix A]{JEMS}. Assumption {\textbf{(A$_3$)}} is also established in \cite[Scetion 7]{JEMS}, where it is obtained by constructing the required smooth controls via shear and cellular flows, and then combining a stability argument with the positivity of Wiener measure. Hence, the remaining key task is to prove the asymptotic strong Feller property for the extended system under mildly degenerate noise, i.e., {\textbf{(A$_2$)}}; this is the \emph{main novelty} of the present work, and the proof is deferred to the next two sections.

Although the results of \cite{NS} cover our setting, our approach avoids Malliavin analysis on the full phase space. By working with a partial Malliavin matrix we obtain a stronger statement, and we also bypass additional, intricate Lie bracket computations, which simplifies the calculations at several points.

\section{Asymptotic gradient estimate}\label{SEC3}
In this section we provide a sufficient condition for the asymptotic strong Feller property of the extended system, namely a long-time asymptotic gradient estimate \eqref{LAGE}. Owing to the compactness of the manifold \(M\) and the super-Lyapunov property (c.f. Corollary \ref{SLP}), it suffices to establish the corresponding short-time asymptotic gradient estimate \eqref{AGE-1}, from which \eqref{LAGE} follows.

\begin{theorem}[Short-time asymptotic gradient estimate] \label{SAGE} 
Consider either the base, Lagrange, or projective process. There exist an $N_*=N_*(\mathcal{E}_0,\nu)$ and a number $\gamma\in(0,1)$ such that the following holds. Set $\mathcal{Z}_0=\{k\in \mathbb{Z}^2,\, 0<|k|\le N_*\}$. For all $\eta>0$, there exist $\tau_0\in [1/2,1)$, $\widetilde{T}_0\in (0,1-\tau_0)$, and hence $T_0:=\tau_0+\widetilde{T}_0\in (1/2,1)$, such that for every $\mathbf{p}_{_0}\in M$ and every Fr\'{e}chet differentiable observables $\varphi:\mathcal{H}\times M\to \mathbb{R}$ with $\|\varphi\|_{\infty}+\|\nabla\varphi\|_{\infty}<\infty$, one has   
    \begin{align} \label{AGE-1}
        \|\nabla \mathcal{P}_{_{T_0}}\varphi(u,\mathbf{p})\|_{\mathcal{H}\times T_{\mathbf{p}} M}\le C\exp{(\eta V(u_0))}\bigg(\sqrt{P_{_{T_0}} |\varphi|^2 (u_0, \mathbf{p}_{_0})}+\gamma\sqrt{P_{_{T_0}}\|\nabla\varphi\|^2_{\mathcal{H}\times T_{\mathbf{p}} M} (u_0, \mathbf{p}_{_0})}\bigg),
    \end{align}
where $C(\eta,\gamma,N_*,\mathbf{p}_{_0})>0$ is locally bounded uniformly in $\mathbf{p}_{_0}$, \(V(u)\) denotes the super-Lyapunov function introduced in Definition \ref{superL}, it is associated with the process \(u_t\) and satisfies the super-Lyapunov property stated in Corollary \ref{SLP}.
\end{theorem}

\begin{remark}
It follows from \cite[Corollary 2.10]{NS} that when the manifold \( M \) is compact, utilizing the super-Lyapunov property (see Corollary \ref{SLP}) of a super-Lyapunov function enables us to translate the short-time asymptotic gradient estimate into the following long-time asymptotic gradient estimate:
\begin{equation} \label{LAGE}
	\|\nabla P_t \varphi\|_{\eta V} \leq C \|\varphi\|_{\eta V} + \delta_t \|\nabla \varphi\|_{\eta V},
\end{equation}
for any $t\ge 1/2$, where $\delta_t:=C(\eta)\gamma^{\lfloor t/T_0\rfloor}$ satisfies $\delta_t\to 0$ as $t\to \infty$, \( C = C(\gamma, \eta,N_*) \) is independent of \( t \) and $\mathbf{p}_0$, and where $$\| \varphi \|_{\eta V} := \mathop{\rm sup}\limits_{u \in H_0^5} \exp(-\eta V(u)) | \varphi(u,\mathbf{p}) |,\;\;\| \nabla\varphi \|_{\eta V} := \mathop{\rm sup}\limits_{u \in H_0^5} \exp(-\eta V(u)) | \nabla\varphi(u,\mathbf{p}) |_{\mathcal{H}\times T_{\mathbf{p}}M }.$$
\end{remark}
\subsection{Malliavin calculus and proof strategy}
The extended system we consider is as follows:
\begin{align}
\dot{u}_t &= - u_t \cdot \nabla u_t+\nu \Delta u_t+\sum_{k \in \mathcal{Z}_0} q_k \nabla^{\perp} \Delta^{-1} e_k(x) \, dW^k_t,\label{u}\\
\dot{x}_t &= u_t(x_t),\label{x}\\
\dot{v}_t &= \Pi_{v_t} Du_t(x_t) v_t.\label{v}
\end{align}
Next, we reformulate the above system as the following abstract stochastic evolution equation on $\mathcal{H}\times \mathbb{T}^2\times S^1$:
\begin{align}\label{abstract system}
\partial_t(u_t,x_t,v_t)=\hat{F}(u_t,x_t,v_t)-\hat{A}(u_t,x_t,v_t)+\hat{Q}\dot{W}_t,
\end{align}
where $\hat{F}$, $\hat{A}$ and $\hat{Q}\dot{W}_t$ are given by 
\[
\hat{F}(u, x, v) =
\begin{pmatrix}
-B(u, u) \\
u(x) \\
\Pi_v Du(x)v 
\end{pmatrix},
\quad \hat{A}(u, x, v) =
\begin{pmatrix}
-\nu \Delta u \\
0 \\
0 
\end{pmatrix},
\quad \hat{Q}\dot{W} =
\begin{pmatrix}
\sum_{k \in \mathcal{Z}_0} q_k \gamma_k e_k(x)\, dW^k_t \\
0 \\
0 
\end{pmatrix}
\]
with $B(u_t,u_t):=u_t \cdot \nabla u_t$.

Before presenting the proof idea of the asymptotic gradient estimates for the extended system, we briefly review the Malliavin calculus preliminaries needed for the subsequent derivations.

Consider the map $Y_t:C([0,t];\mathbb{R}^{m})\times \mathcal{H}\times M \to \mathcal{H}\times M$ such that $(u_t,\mathbf{p}_t) = Y_t(W,u_0,\mathbf{p}_0)$ for initial data $u_0,\,\mathbf{p}_0$ and noise realization $W$, where $m:=|\mathcal{Z}_0|$. Take a direction \( g \in L^2([0,T_0];\mathbb{R}^{m}) \), and let 
$$
G(t) = \int_0^t g_{_{s}} \, ds.
$$
For an \(\mathcal{H}\times M\)-valued random variable \((u_t,\mathbf{p}_t)\), its Malliavin derivative in the direction \(g\) is defined as 
$$
\mathcal{D}^g (u_t,\mathbf{p}_t) = \lim_{\varepsilon \to 0} \frac{Y_t(W + \varepsilon G, u_0,\mathbf{p}_0) - Y_t(W, u_0,\mathbf{p}_0)}{\varepsilon}\,,
$$
where the limit holds almost surely with respect to the Wiener measure. Note that we allow \(g\) to be random and possibly nonadapted to the filtration generated by the increments of \(W\).
The Fr\'{e}chet differentiability of \(Y_t\) further yields
\begin{align}\label{JEMS-6.6}
\mathcal{D}^g (u_t,\mathbf{p}_t) = \int_{0}^{T_0}\mathcal{D}_s(u_t,\mathbf{p}_t)g_{_{s}} \, ds,
\end{align}
where \(\mathcal{D}_s(u_t,\mathbf{p}_t)\) denotes the Malliavin derivative of \((u_t,\mathbf{p}_t)\) at time \(s\). We also write \(\mathcal{J}_{s,t}\) (\(s\le t\)) for the derivative flow from \(s\) to \(t\): for any \(\mathfrak{h}\in \mathcal{H}\times T_{\mathbf{p}_s}M\), \(\mathcal{J}_{s,t}\mathfrak{h}\) is the solution of
\begin{align} \label{Jh1}
\partial_t \mathcal{J}_{s,t}\mathfrak{h} = D\hat{F}(u_t,\mathbf{p}_t)\mathcal{J}_{s,t}\mathfrak{h}-\hat{A}\mathcal{J}_{s,t}\mathfrak{h}, \quad \mathcal{J}_{s,s}\mathfrak{h} = \mathfrak{h}.
\end{align}
By the Duhamel's formula, we obtain
$$
\mathcal{D}^g (u_t,\mathbf{p}_t) = \int_0^{T_0} \mathcal{J}_{s,t}\hat{Q}g_{_{s}} \, ds.
$$
Combining this with \eqref{JEMS-6.6}, we arrive at the following observation
\begin{align}\label{6.7}
\mathcal{D}_s (u_t,\mathbf{p}_t) = 
\begin{cases}
\mathcal{J}_{s,t} \hat{Q}, & s < t, \\
0, & s > t.
\end{cases}
\end{align}

We now outline the proof idea for the asymptotic gradient estimates.

For $\mathfrak{h} \in \mathcal{H}\times T_{\mathbf{p}} M$, we use a bracket to denote the inner product with respect to \( \mathcal{H} \times T_\mathbf{p}M \), that is
\[
\left\langle{\begin{pmatrix} \varphi \\ \mathbf{p} \end{pmatrix}},{\begin{pmatrix} \psi \\ \mathbf{q} \end{pmatrix}}\right\rangle := \hat{g}(\mathbf{p},\mathbf{q}) + \left\langle{\varphi},{\psi}\right\rangle_{\mathcal{H}},
\]
then $\|\mathfrak{h}\|:=\|\mathfrak{h}\|_{\mathcal{H}\times T_{\mathbf{p}} M}=\sqrt{\left\langle{\mathfrak{h}},\mathfrak{h}\right\rangle}$. Now let $\mathfrak{h}\in \mathcal{H}\times T_{\mathbf{p}} M$ with $\|\mathfrak{h}\|=1$ arbitrary. Set $\rho_{_t}=\mathcal{J}_{0,t}\mathfrak{h}-\mathcal{D}^{g}$, we then have the fundamental approximate integration by parts computation,
\begin{align}\label{3.9}
&\left\langle\nabla P_{_{T_0}} \varphi(u_0,\mathbf{p}_{_0}),\mathfrak{h}\right\rangle\notag\\
&=\mathbb{E}_{(u_0,\mathbf{p}_{_0})}((\nabla\varphi)(u_t,\mathbf{p}_t)\mathcal{J}_{0,T_0}\mathfrak{h})\notag\\
&=\mathbb{E}_{(u_0,\mathbf{p}_{_0})}(\mathcal{D}^{g}\varphi(u_t,\mathbf{p}_t))+\mathbb{E}_{(u_0,\mathbf{p}_{_0})}((\nabla\varphi)(u_t,\mathbf{p}_t)\rho_{_{T_0}})\notag\\
&=\mathbb{E}_{(u_0,\mathbf{p}_{_0})}(\varphi(u_{_{T_0}},\mathbf{p}_{_{T_0}})\int_0^{T_0}g_{_s}dW_s)+\mathbb{E}_{(u_0,\mathbf{p}_{_0})}((\nabla\varphi)(u_t,\mathbf{p}_t)\rho_{_{T_0}})\notag\\
&\le \left(\mathbb{E}_{(u_0,\mathbf{p}_{_0})}\left(\int_0^{T_0}g_{_s}dW_s\right)^2\right)^{1/2}\sqrt{P_{_{T_0}} |\varphi|^2 (u_0, \mathbf{p}_{_0})}+\left(\mathbb{E}_{(u_0,\mathbf{p}_{_0})}\|\rho_{_{T_0}}\|_{\mathcal{H}}^2\right)^{1/2}\sqrt{P_{_{T_0}} \|\nabla\varphi\|^2_{\mathcal{H}\times T_{\mathbf{p}} M} (u_0, \mathbf{p}_{_0})}.
\end{align}
The core idea of the proof is to construct an infinitesimal perturbation $g_t$ of the Wiener process path, such that its induced effect on the system state $(u_t, \mathbf{p}_{_t})$ compensates for, as much as possible, the effect caused by an infinitesimal perturbation $\mathfrak{h}$ of the initial condition at time $t$. Subsequently, our focus will be on identifying a suitable control $g_{_t}$ such that there exists $\gamma\in (0,1)$ such that the associated control cost and resulting error satisfy
\begin{equation}\label{erro-estimate1}
\sup_{\|\mathfrak{h}\|_{\mathcal{H}\times T_{\mathbf{p}}M}=1} \mathbb{E}_{(u_0,\mathbf{p}_{_0})} \| \rho_{_{T_0}}\|_{\mathcal{H}}^2 \leq \gamma\exp(\eta V(u_0))
\end{equation}
and
\begin{equation}\label{v-estimate1}
\sup_{\|\mathfrak{h}\|_{\mathcal{H}\times T_{\mathbf{p}}M}=1} \mathbb{E}_{(u_0,\mathbf{p}_{_0})} \left| \int_{0}^{T_0} g_{_s} dW_s \right|^2 \leq C \exp(\eta V(u_0)).
\end{equation}

\subsection{Splitting of extended system}
Here, exploiting the mild degeneracy of the noise, we implement a natural splitting of the above extended system into high and low frequencies. This splitting will subsequently allow us to construct controls within the \emph{finite-dimensional} low-frequency subsystem, while fully exploiting the inherent dissipative effects present in the high-frequency component. Here, $\mathcal{Z}_0:=\{k\in \mathbb{Z}^2,\, 0<|k|\le N_*\}$ is the set of low (frequency) modes. Let \(\Pi_l : \mathcal{H} \to \mathcal{H}\) denote the corresponding orthogonal projection onto the `low modes' belonging to \(\mathcal{Z}_0\) and let \(\Pi_h = I - \Pi_l\) be the complementary projection onto the `high modes' belonging to \(\mathbb{Z}^2\setminus \mathcal{Z}_0\). Let \(\mathcal{H}_l(\supseteq Rang(Q))\) and \(\mathcal{H}_h\) denote the ranges of \(\Pi_l\) and \(\Pi_h\), respectively, so that we have the orthogonal decomposition
\[
\mathcal{H} = \mathcal{H}_l \oplus \mathcal{H}_h.
\]
Given \((u, x, v) \in \mathcal{H} \times M\), we will extend the definition of \(\Pi_l\) and \(\Pi_h\) to \(\mathcal{H} \times M\) so that \(M\) is included with the low modes by
\[
(u, x, v)^l = \Pi_l (u, x, v) = (\Pi_lu, x, v)=(u^l, x, v) \quad \text{and} \quad (u, x, v)^h = \Pi_h (u, x, v) = \Pi_h u = u^h.
\]

Naturally this defines low and high processes \((u, x, v)^l\) and \((u, x, v)^h\), which satisfy (note of course they are \emph{coupled})

\[
\begin{aligned}
\partial_t (u_t, x_t, v_t)^l &= \hat{F}_l(u_t, x_t, v_t) - \hat{A}_l (u_t, x_t, v_t)^l +\hat{Q}_l \dot{W}_t^l, \\
\partial_t (u_t, x_t, v_t)^h &= \hat{F}_h(u_t, x_t, v_t) - \hat{A}_h (u_t, x_t, v_t)^h + \hat{Q}_h \dot{W}_t^h,
\end{aligned}
\]
where 
\[
\hat{F}_l(u_t, x_t, v_t) = \Pi_l \hat{F}(u_t, x_t, v_t), \quad \hat{F}_h(u_t, x_t, v_t) = \Pi_h \hat{F}(u_t, x_t, v_t), \quad \hat{A}_l(u_t, x_t, v_t)  = \Pi_l \hat{A}(u_t, x_t, v_t),\] 
\[
\hat{A}_h(u_t, x_t, v_t)  = \Pi_h \hat{A}(u_t, x_t, v_t), \quad \hat{Q}_l=\Pi_l\hat{Q}, \quad \hat{Q}_h=\Pi_h\hat{Q}.
\]

Recall that for any \(\mathfrak{h} \in \mathcal{H} \times TM\), we denote its norm by  
\(\|\mathfrak{h}\| := \|\mathfrak{h}\|_{\mathcal{H} \times TM}\).  
Similarly, for any \(\mathfrak{h}^l \in \mathcal{H}_l \times TM\) and $\mathfrak{h}^h\in \mathcal{H}_h$, we define  
\(\|\mathfrak{h}^l\|_l := \|\mathfrak{h}^l\|_{\mathcal{H}_l \times TM}\) and $\|\mathfrak{h}^h\|_h := \|\mathfrak{h}^h\|_{\mathcal{H}_h}$.  In addition, we use \(\langle \cdot,\cdot \rangle_l\) to denote the inner product on \(\mathcal{H}_l \times TM\), and \(\langle \cdot,\cdot \rangle_h\) to denote the inner product on \(\mathcal{H}_h\).

\begin{remark}
To fully exploit the dissipative effect of the high modes, we incorporate the manifold component $(x_t,v_t)$ into the low-frequency subsystem, utilizing its finite-dimensional structure to design the control strategy and thereby effectively overcoming the non-invertibility of \(\hat{Q}_l\). This design achieves a logical separation and coordination between control and dissipation, ensuring that controllability and stability are simultaneously realized within a unified analytical framework.

\end{remark}

\subsection{Construction and cost of control}
We introduce two operators: the finite-dimensional matrix \( R_{s,t}^l \), viewed as a linear map from \( \mathcal{H}_l \times T_{\mathbf{p}_s}M \) to \( \mathcal{H}_l \times T_{\mathbf{p}_t}M \), and the bounded linear operator \( R_{s,t}^h \) from \( \mathcal{H}_h \) to \( \mathcal{H}_h \). These operators are defined by the following evolution equations:
\begin{align} \label{Rl1}
\partial_t R_{s,t}^l = -\hat{A}_l R_{s,t}^l + D_l \hat{F}_l (u_t,\mathbf{p}_t) R_{s,t}^l, \quad R_{s,s}^l = \text{Id},
\end{align}
and, for \( 0 \leq s \leq t \),
\begin{align} \label{Rh1}
\partial_t R_{s,t}^h = -\hat{A}_h R_{s,t}^h + D_h \hat{F}_h (u_t,\mathbf{p}_t) R_{s,t}^h, \quad R_{s,s}^h = \text{Id}.
\end{align}
More specifically, the above evolution equation can also be written as:
\begin{equation} \label{PL}
\begin{cases}
\partial_t R_{s,t}^l = \tilde{L}_t^l R_{s,t}^l, \\
R_{s,s}^l = \text{Id},
\end{cases}
\quad \text{and} \quad
\begin{cases}
\partial_t R_{s,t}^h = \tilde{L}_t^h R_{s,t}^h, \\
R_{s,s}^h = \text{Id},
\end{cases}
\end{equation}
where \( \tilde{L}_t^l := \tilde{L}^l(u_t, x_t, v_t)= \tilde{L}^l(u_t, \mathbf{p}_{_t}) \colon H^7_l \times T_{\mathbf{p}_{_t}}M \to H^5_l \times T_{\mathbf{p}_{_t}}M \) is given by
\[
\tilde{L}_t^l \begin{pmatrix} \mathfrak{h}_u^l \\ \mathfrak{h}_x \\ \mathfrak{h}_v \end{pmatrix} := 
\begin{pmatrix}
	\nu \Pi_l(\Delta \mathfrak{h}_u^l) - \Pi_l(\mathfrak{h}_u^l\cdot \nabla u_t+u_t\cdot \nabla\mathfrak{h}_u^l) \\
	Du(x)\mathfrak{h}_x+ \mathfrak{h}_u^l(x)\\
  D[\Pi_vDu(x)v]\mathfrak{h}_v+\Pi_vDu(x)\mathfrak{h}_v+\Pi_v[D^2u(x)\mathfrak{h}_x]v+\Pi_vD\mathfrak{h}_u^l(x)v
\end{pmatrix},
\]
and \( \tilde{L}_t^h := \tilde{L}^h(u_t) \colon H^7_h \to H^5_h \) is given by
\[
\tilde{L}_t^h(\mathfrak{h}^h):=\nu \Pi_h(\Delta \mathfrak{h}^h) - \Pi_h(\mathfrak{h}^h\cdot \nabla u_t+u_t\cdot \nabla\mathfrak{h}^h).
\]
Both \( R_{s,t}^l \) and \( R_{s,t}^h \) provide approximations to the projection of the full Jacobian \( J_{s,t} \) (the derivative of the process $(u_t, x_t, v_t)$) onto the low and high modes, respectively, when the time \( t \) is sufficiently small. Observe that \( R_{s,t}^l \) is invertible (due to its satisfaction of a finite-dimensional linear evolution equation), then we denote its inverse by
\[
S_{s,t}^l = (R_{s,t}^l)^{-1}.
\]
Set \(T_0 = \widetilde{T}_0 + \tau_0(<1)\), where \(\widetilde{T}_0 \ll 1\) and $\tau_0\in[1/2,1)$ are to be determined later. We simplify the notation by writing \( R_{t}^l = R_{\tau_0,t}^l \) and \( S_{t}^l = S_{\tau_0,t}^l \). When $\tau_0\le s<t\le T_0$, by virtue of the invertibility of \( R_{t}^l \), we can express \( R_{s,t}^l \) as \( R_{s,t}^l = R_{t}^l S_{s}^l \).


We first introduce the key notion of the partial Malliavin matrix, which will play a central role in the subsequent control design.
\begin{definition}\label{PM}
    Define the partial Malliavin matrix \( \mathcal{N}_{_{T_0}}^l: \mathcal{H}_l \times T_{\mathbf{p}_{_{\tau_0}}} M \to \mathcal{H}_l \times T_{\mathbf{p}_{_{T_0}}} M \) by
\[
\mathcal{N}_{_{T_0}}^l = \int_{\tau_0}^{T_0} \,S_s^l\hat{Q}^l\,(S_s^l\hat{Q}^l)^\top\, ds.
\]
\end{definition}
\begin{remark}
The matrix \(\mathcal{N}_{_{T_0}}^l\) corresponds to a reduced Malliavin matrix, originally introduced by Norris \cite{JEMS-94} to streamline the Malliavin’s proof of H\"{o}rmander’s theorem. The term `partial Malliavin matrix' is adopted from \cite{JFA-42} and further employed in \cite{HM06,HM11a}; it refers specifically to the \emph{finite-dimensional Malliavin matrix associated with the low modes}. Here, we note that the usual definition of the Malliavin matrix would be
\[
\mathcal{N}_{_{T_0}}^l = \int_{\tau_0}^{T_0} \,R_{s,T_0}^l\hat{Q}^l\,(R_{s,T_0}^l\hat{Q}^l)^\top\, ds,
\]
but for such a reduced matrix, we exploit the invertibility of \(R_{s,T_0}^l\) to define it here as in Definition \ref{PM}.

The main advantage of this formulation is that, unlike the original Malliavin matrix, \(\mathcal{N}_{_{T_0}}^l\) is adapted, which greatly simplifies calculations. For instance, the It\'{o}'s formula can be applied directly to generate Lie brackets, as detailed in \cite[Lemma 6.15]{JEMS}. Moreover, the invertibility of \(R_{s,T_0}^l\) further justifies the frequency separation between high and low modes.
\end{remark}

Indeed, we can later prove the non-degeneracy of \(\mathcal{N}_{_{T_0}}^l\), which allows us to construct the low‑frequency component of the control \(g_{_{t}},\,t\in[\tau_0,T_0]\). More precisely, we have the following result; the proof is somewhat involved and is postponed to Section \ref{SEC4} below.
\begin{proposition} \label{M-1}
The matrix \( \mathcal{N}_{_{T_0}}^l \) is almost surely invertible on \( \mathcal{H}_l \times T_{\mathbf{p}_{_{T_0}}}M \). Furthermore, for all \( p \geq 1 \), $\eta\in (0,1)$ and $\mathbf{p}_{_0}\in M$ there exists $C(N_*,\nu,\mathbf{p}_{_0},\eta,p)>0$, locally bounded in $\mathbf{p}_{_0}$, such that 
\[
\mathbb{E}|(\mathcal{N}_{_{T_0}}^l)^{-1}|^p \leq C(N_*,\nu,\mathbf{p}_{_0},\eta,p) \exp{(p\eta V(u_0))}.
\]
\end{proposition}
We will then use Proposition \ref{M-1} to construct a perturbation \( g^l \) which is given by 0 on all intervals of the type \( [0,\tau_0] \), and by \( g_t^l \in L^2([\tau_0,T_0];{\mathbb{R}^{m}}) \) on the remaining intervals. 

We define the infinitesimal variation \( g_t^l \) by
\begin{equation}\label{eq:4.15}
g_t^l = (S_t^l\hat{Q}_l)^\top(\mathcal{N}_{_{T_0}}^l)^{-1}S_{_{T_0}}^l\Pi_l(J_{\tau_0,T_0}(J_{0,\tau_0}\mathfrak{h}))=(S_t^l\hat{Q}_l)^\top(\mathcal{N}_{_{T_0}}^l)^{-1}S_{_{T_0}}^l\Pi_l(J_{0,T_0}\mathfrak{h}).
\end{equation}

Thereafter, with a  slightly abuse notation we will use \(g_t^l\) to denote both the perturbation of the Wiener path on \([\tau_0,T_0]\) and its extension (by zero) to the interval \([0,\tau_0]\). Note that here we leave the high-frequency part uncontrolled, hence the overall control is given by 
\begin{align} \label{g-t}
g_t = (g_t^l,\, 0)=\left((S_t^l\hat{Q}_l)^\top(\mathcal{N}_{_{T_0}}^l)^{-1}S_{_{T_0}}^l\Pi_l(J_{0,T_0}\mathfrak{h}),\,0\right).
\end{align}

We next derive an estimate for the control cost.
\begin{proposition}[Cost of control]\label{v-estimate-p}
There exists a constant $C(N_*,\nu,\mathbf{p}_0)$, locally bounded in $\mathbf{p}_0$, such that
\begin{equation}\label{v-estimate}
\mathbb{E}\left| \int_{\tau_0}^{T_0} g_{_s} dW_s \right|^2 \leq C \exp(\eta V(u_0))\|\mathfrak{h}\|_{\mathcal{H}\times T_{\mathbf{p}}M}.
\end{equation}
\end{proposition}
This argument is highly similar to \cite[Section 4.6]{HM06} and requires no new ideas. Here, the control $g_t$ is not adapted, estimating it requires appealing to the fundamental $L^2$-isometry of the Skorokhod integral.
\begin{proposition}\emph{\cite[Proposition 1.3.1]{Jems-95}} \label{S}
Let \( g \in \mathbb{W}^{1,2}(\Omega_T; L^2([0, T]; {\mathbb{R}^{m}}))  \). Then the following identity holds:
\begin{align*}
\mathbb{E} \left( \int_0^T \langle g_{_t}, \delta W_t \rangle_{{\mathbb{R}^{m}}} \right)^2
&= \mathbb{E} \int_0^T \|g_{_s}\|_{{\mathbb{R}^{m}}}^2 \, ds + \mathbb{E} \int_0^T \int_0^T \operatorname{tr}_{_{\mathbb{R}^{m}}} (\mathcal{D}_s g_{_t} \mathcal{D}_t g_{_s}) \, ds \, dt\\
&\leq \mathbb{E} \int_0^T \|g_{_s}\|_{{\mathbb{R}^{m}}}^2 \, ds + \mathbb{E} \int_0^T \int_0^T \|\mathcal{D}_t g_{_s}\|_{{\mathbb{R}^{m}} \to {\mathbb{R}^{m}}}^2 \, ds \, dt \\
&= \|g\|_{\mathbb{W}^{1,2}(\Omega_T; L^2([0, T]; {\mathbb{R}^{m}}))}^2,
\end{align*}
where $\mathcal{D}_s$ is the Malliavin derivative of the noise at time $s$ as in \eqref{JEMS-6.6}.
\end{proposition}

\begin{proof}[Proof of Proposition \ref{v-estimate-p}]
By Proposition \ref{S}, we immediately obtain
$$
\mathbb{E}\left| \int_{\tau_0}^{T_0} g_{_s} dW_s \right|^2 \le \mathbb{E} \|g_{_t}\|_{L^2([\tau_0,T_0];{\mathbb{R}^{m}})}^2 + \mathbb{E} \int_{\tau_0}^{T_0} \|\mathcal{D}_s g_{_t}\|_{L^2([\tau_0,T_0];{\mathbb{R}^{m}})}\, ds.
$$

Firstly, we note that
\begin{align*}
\|g_{_t}\|_{L^2([\tau_0,T_0])}&=\|g_{_t}^l\|_{L^2([\tau_0,T_0])}=\|(S_t^l\hat{Q}_l)^\top(\mathcal{N}_{_{T_0}}^l)^{-1}S_{_{T_0}}^l\Pi_l(J_{0,T_0}\mathfrak{h})\|_{L^2([\tau_0,T_0])} \\
&\leq \|(S_t^l\hat{Q}_l)^\top\|_{\mathcal{H}_l\times TM \to L^2([\tau_0,T_0])} \,\,|\mathcal{N}_{_{T_0}}^l|^{-1} \,\,\|S_{_{T_0}}^l\|_{\mathcal{H}_l\times TM \to \mathcal{H}_l\times TM}\,\,\|\Pi_l(J_{0,T_0}\mathfrak{h})\|_{\mathcal{H}_l\times TM}. 
\end{align*}
Note also that by the product rule, we have
\begin{align*}
&\|\mathcal{D}_s g_{_t}\|_{L^2([\tau_0,T_0])} \\
&=\|\mathcal{D}_s\left((S_t^l\hat{Q}_l)^\top(\mathcal{N}_{_{T_0}}^l)^{-1}S_{_{T_0}}^l\Pi_l(J_{0,T_0}\mathfrak{h})\right)\|_{L^2([\tau_0,T_0])}\\
&\le \|\mathcal{D}_s\big((S_t^l\hat{Q}_l)^\top\big)\|_{\mathcal{H}_l\times TM \to L^2([\tau_0,T_0])}\,\,|\mathcal{N}_{_{T_0}}^l|^{-1} \,\,\|S_{_{T_0}}^l\|_{\mathcal{H}_l\times TM \to \mathcal{H}_l\times TM}\,\,\|\Pi_l(J_{0,T_0}\mathfrak{h})\|_{\mathcal{H}_l\times TM}\\
&\quad +\|(S_t^l\hat{Q}_l)^\top\big\|_{\mathcal{H}_l\times TM \to L^2([\tau_0,T_0])}\,\,\|\mathcal{D}_s(\mathcal{N}_{_{T_0}}^l)^{-1}\|_{L^2([\tau_0,T_0])\to \mathcal{H}_l\times TM} \,\,\|S_{_{T_0}}^l\|_{\mathcal{H}_l\times TM \to \mathcal{H}_l\times TM}\,\,\|\Pi_l(J_{0,T_0}\mathfrak{h})\|_{\mathcal{H}_l\times TM}\\
&\quad +\|(S_t^l\hat{Q}_l)^\top\big\|_{\mathcal{H}_l\times TM \to L^2([\tau_0,T_0])}\,\,|\mathcal{N}_{_{T_0}}^l|^{-1} \,\,\|\mathcal{D}_sS_{_{T_0}}^l\|_{\mathcal{H}_l\times TM \to \mathcal{H}_l\times TM}\,\,\|\Pi_l(J_{0,T_0}\mathfrak{h})\|_{\mathcal{H}_l\times TM}\\
&\quad +\|(S_t^l\hat{Q}_l)^\top\big\|_{\mathcal{H}_l\times TM \to L^2([\tau_0,T_0])}\,\,|\mathcal{N}_{_{T_0}}^l|^{-1} \,\,\|S_{_{T_0}}^l\|_{\mathcal{H}_l\times TM \to \mathcal{H}_l\times TM}\,\,\|\mathcal{D}_s\big(\Pi_l(J_{0,T_0}\mathfrak{h})\big)\|_{\mathcal{H}_l\times TM}
\end{align*}
Note also that the following holds:
\begin{align} \label{DN}
\mathcal{D}_s(\mathcal{N}_{_{T_0}}^l)^{-1}f=-(\mathcal{N}_{_{T_0}}^l)^{-1}(\mathcal{D}_s\mathcal{N}_{_{T_0}}^lf)(\mathcal{N}_{_{T_0}}^l)^{-1},\qquad 
\mathcal{D}_sS_{_{T_0}}^lf=-S_{_{T_0}}^l(\mathcal{D}_sR_{_{T_0}}^lf)S_{_{T_0}}^l.
\end{align}
Thus Proposition \ref{v-estimate-p} follows from the Proposition \ref{M-1}, \eqref{DN}, \eqref{Sl}, \eqref{J} together with the following moment bounds, where
$C(N_*,\nu,\mathbf{p}_0, \eta) > 0$ is locally bounded in $\mathbf{p}_0$,
\begin{align}
\mathbb{E} \sup_{\tau_0 \leq r< t \leq T_0} \left( \|\mathcal{D}_s R_{r,t}^l\mathfrak{h}^l\|_{L^2([\tau_0,T_0])\to \mathcal{H}_l\times T_{\mathbf{p}}M}\right) &\le C\, \widetilde{T}_0\,e^{\eta V(\mathbf{u}_0)}\|\mathfrak{h}^l\|_l, \label{DR}\\
\mathbb{E} \sup_{\tau_0 \leq r< t \leq T_0} \|\mathcal{D}_s J_{r,t}\mathfrak{h}\|_{L^2([\tau_0,T_0])\to \mathcal{H}\times T_{\mathbf{p}}M} &\le C \,\widetilde{T}_0^{1/2}\,e^{\eta V(\mathbf{u}_0)}\|\mathfrak{h}\|. \label{DJ}
\end{align}

It follows from \eqref{Rl1} that $R_{r,t}^l\mathfrak{h}^l$ satisfies
\begin{align*} 
\partial_t R_{r,t}^l\mathfrak{h} = -\hat{A}_l R_{r,t}^l\mathfrak{h} + D_l \hat{F}_l (u_t,\mathbf{p}_t) R_{r,t}^l\mathfrak{h}, \quad R_{r,r}^l\mathfrak{h} = \mathfrak{h}.
\end{align*}
Differentiating the above equation in the Wiener path for fixed \(\mathfrak{h}\) and applying the chain rule gives:
\begin{align*} 
\partial_t (\mathcal{D}_s(R_{r,t}^l\mathfrak{h})) &= (-\hat{A}_l+ D_l \hat{F}_l (u_t,\mathbf{p}_t)) \mathcal{D}_s(R_{r,t}^l\mathfrak{h})+\overline{D}^2_l\hat{F}_l(\mathcal{D}_s(u_t,\mathbf{p}_t),R_{r,t}^l\mathfrak{h})\\
&= (-\hat{A}_l+ D_l \hat{F}_l (u_t,\mathbf{p}_t)) \mathcal{D}_s(R_{r,t}^l\mathfrak{h})+\overline{D}^2\hat{F}_l(J_{s,t}\hat{Q},R_{r,t}^l\mathfrak{h}),
\end{align*}
where $\overline{D}^2\hat{F}$ denotes the full second variation of $F$ extended to the linear space $\mathcal{H}_l\times \mathbb{R}^{2d}$, and the last inequality employs the observation \eqref{6.7}.
Then, by the variation of constants formula, we can write it in the following integral form:
\begin{align*} 
\mathcal{D}_s(R_{r,t}^l\mathfrak{h})f=\int_r^t R_{\tau,t}^l\left(\overline{D}^2\hat{F}_l(J_{s,\tau}\hat{Q}f,R_{r,\tau}^l\mathfrak{h})\right)\, d \tau.
\end{align*}
Analogously, we obtain
\begin{align*} 
\mathcal{D}_s(J_{r,t}\mathfrak{h})f=\int_r^t J_{\tau,t}\left(\overline{D}^2\hat{F}(J_{r,\tau}\hat{Q}f,J_{r,\tau}\mathfrak{h})\right)\, d \tau.
\end{align*}

To prove \eqref{DR}–\eqref{DJ}, we only need to show
$$
\mathbb{E} \sup_{\tau_0 \leq t \leq T_0} \|\overline{D}^2\hat{F}(u_t,\mathbf{p}_{_t})\|_{ (\mathcal{H}_l\times \mathbb{R}^{2d})\bigotimes(\mathcal{H}_l\times \mathbb{R}^{2d})}\le C(N_*,\nu)e^{\eta V(u_0)}.
$$
Fix $(u,x,v)$, with \(\mathfrak{h}=(\mathfrak{h}^u,\mathfrak{h}^x,\mathfrak{h}^v)\) and \(\mathfrak{k}=(\mathfrak{k}^u,\mathfrak{k}^x,\mathfrak{k}^v)\) being tangent vectors, if we denote \(\hat{F}^u=B(u,u)\), it follows that
\begin{align*}
    \|\overline{D}^2\hat{F}^u(u)[\mathfrak{h},\mathfrak{k}]\|_{\mathcal{H}_l}\le \|B(\mathfrak{h}^u,\mathfrak{k}^u)\|_{\mathcal{H}_l}+\|B(\mathfrak{k}^u,\mathfrak{h}^u)\|_{\mathcal{H}_l}\le C(N_*)\|\mathfrak{h}^u\|_{\mathcal{H}_l}\|\mathfrak{k}^u\|_{\mathcal{H}_l},
\end{align*}
then
\begin{align*}
\|\overline{D}^2\hat{F}^u(u)\|_{\mathcal{L}^2(\mathcal{H}_l,\mathcal{H}_l)}\le C(N_*).
\end{align*}
Set $\hat{F}^x=u(x)$, computation yields
$$
\overline{D}^2\hat{F}^x(u,x)[\mathfrak{h},\mathfrak{k}]=\overline{D}\mathfrak{h}^u(x)\mathfrak{k}^x+D\mathfrak{k}^u(x)\mathfrak{h}^x+\overline{D}^2u(x)[\mathfrak{h}^x,\mathfrak{k}^x].
$$
Since the space is finite-dimensional, combined with \eqref{NS-2.5} we obtain
\begin{align*}
\mathbb{E}\sup_{\tau_0 \leq t \leq T_0}\|\overline{D}^2\hat{F}^x(u_t,x_t)\|_{\mathcal{L}^2(\mathcal{H}_l\times \mathbb{R}^{d},\mathcal{H}_l\times \mathbb{R}^{d})}\le C(N_*,\nu)e^{\eta V(u_0)}.
\end{align*}
Similarly, setting $\hat{F}^v=\Pi_v Du(x)v$, we have 
\begin{align*}
\mathbb{E}\sup_{\tau_0 \leq t \leq T_0}\|\overline{D}^2\hat{F}^v(u_t,x_t,v_t)\|_{\mathcal{L}^2(\mathcal{H}_l\times \mathbb{R}^{2d},\mathcal{H}_l\times \mathbb{R}^{2d})}\le C(N_*,\nu)e^{\eta V(u_0)}.
\end{align*}
This completes the proof.
\end{proof}
\subsection{Error representation and estimates}

Due to the coupling between high and low frequencies present in both the nonlinear terms and the vector fields on the manifold, the high and low-frequency components of the Malliavin derivative are themselves typically coupled. First, based on the evolution equation for the Malliavin derivative $\mathcal{D}^{g}(u_{t},\mathbf{p}_{_t})$,we have
\begin{align}\label{Dg}
\partial_t(\mathcal{D}^{g}(u_{t},\mathbf{p}_{_t}))=-\hat{A}\mathcal{D}^{g}(u_{t},\mathbf{p}_{_t})+D\hat{F}(u_t,\mathbf{p}_{_t})\mathcal{D}^{g}(u_{t},\mathbf{p}_{_t})+\hat{Q}g_{_t},\, \, \, \,\, \, \, \,\, \mathcal{D}^{g_{_t}}(u_{0},\mathbf{p}_{_0})=0.
\end{align}
Set
$$
\mathcal{D}^{g}(u_{t},\mathbf{p}_{_t})=\left(\mathcal{D}^{g}(u^l_{t},\mathbf{p}_{_t}),\,\mathcal{D}^{g}(u_{t}^h)\right):=(\zeta_t,\,\xi_t).
$$
Then, for each $t\in [\tau_0,T_0]$, $(\zeta_t,\,\xi_t)$ solve the following system:
\begin{equation} \label{xi-zeta}
\begin{cases}
\dot{\zeta}_t = -\hat{A}_l \zeta_t + D_l \hat{F}_l (u_t,\mathbf{p}_{_t}) \zeta_t + D_h \hat{F}_l (u_t,\mathbf{p}_{_t}) \xi_t+ \hat{Q}_l g^l_{_t}, \\
\dot{\xi}_t = -\hat{A}_h \xi_t + D_l \hat{F}_h (u_t,\mathbf{p}_{_t}) \zeta_t + D_h \hat{F}_h (u_t,\mathbf{p}_{_t}) \xi_t,
\end{cases}
\end{equation}
with \(\zeta_{\tau_0} = 0\) and \(\xi_{\tau_0} = 0\). Moreover, one can obtain the expression for the error given in the following lemma.

\begin{lemma}\label{error-express}
Assume that \( g_t \)  and \( (\zeta_t, \xi_t) \) are defined as above, then the remainder  
\[
\rho_{_{T_0}}= J_{0,T_0}\mathfrak{h}-\mathcal{D}^{g} (u_{_{T_0}},\mathbf{p}_{_{T_0}})
\]  
satisfies  
\begin{align}
\rho_{_{T_0}}^l &= -\int_{\tau_0}^{T_0} R_{s,T_0}^l D_h \hat{F}_l(u_s,\mathbf{p}_{_s}) \xi_s \, ds, \label{eq:rT0L} \\
\rho_{_{T_0}}^h &= R_{\tau_0,T_0}^h(\Pi_h(J_{0,\tau_0}\mathfrak{h}))+\int_{\tau_0}^{T_0}R_{s,T_0}^hD_l\hat{F}_h(u_s,\mathbf{p}_{_s})\Pi_l(J_{0,s}\mathfrak{h})\, ds-\int_{\tau_0}^{T_0} R_{s,T_0}^hD_l\hat{F}_h(u_s,\mathbf{p}_{_s})\zeta_s\, ds.\label{eq:rT0H}
\end{align}
\end{lemma}

\begin{proof}
    Using \eqref{Dg} and \eqref{xi-zeta}, we obtain that in fact \(\mathcal{D}^{g}(u^l_{t},\mathbf{p}_{_t}) = \zeta_t\) and \(\mathcal{D}^{g}(u^h_{t}) = \xi_t\) and we obtain the following formulas for the Malliavin derivatives at time \(T_0\): 
\begin{equation*} 
\begin{cases}
\partial_t\mathcal{D}^{g}(u^l_{t},\mathbf{p}_{_t}) = -\hat{A}_l \mathcal{D}^{g}(u^l_{t},\mathbf{p}_{_t}) + D_l \hat{F}_l (u_t,\mathbf{p}_{_t}) \mathcal{D}^{g}(u^l_{t},\mathbf{p}_{_t}) + D_h \hat{F}_l (u_t,\mathbf{p}_{_t}) \mathcal{D}^{g}(u^h_{t})+ \hat{Q}_l g^l_{_t}, \\
\partial_t\mathcal{D}^{g}(u^h_{t}) = -\hat{A}_h \mathcal{D}^{g}(u^h_{t}) + D_l \hat{F}_h (u_t,\mathbf{p}_{_t}) \mathcal{D}^{g}(u^l_{t},\mathbf{p}_{_t}) + D_h \hat{F}_h (u_t,\mathbf{p}_{_t}) \mathcal{D}^{g}(u^h_{t}),
\end{cases}
\end{equation*}
with \(\mathcal{D}^{g}(u^l_{{\tau_0}},\mathbf{p}_{_{\tau_0}}) = 0\) and \(\mathcal{D}^{g}(u^h_{{\tau_0}}) = 0\). By the variation of constants formula and the invertibility of $\mathcal{N}_{_{T_0}}^l$, we obtain
\begin{align*}
    \mathcal{D}^{g}(u^l_{T_0},\mathbf{p}_{_{T_0}})&= \Pi_l({J_{0,T_0}\mathfrak{h}})+\int_{\tau_0}^{T_0}R_{s,T_0}^l D_h \hat{F}_l(u_s,\mathbf{p}_{_s}) \xi_s \, ds,\\
    \mathcal{D}^{g}(u^h_{T_0})&= \int_{\tau_0}^{T_0} R_{s,T_0}^hD_l\hat{F}_h(u_s,\mathbf{p}_{_s})\zeta_s\, ds.
\end{align*}
Using this relation, we now write
\[
\begin{cases}
\Pi_l({J_{0,T_0}\mathfrak{h}}) = \mathcal{D}^{g}(u^l_{T_0},\mathbf{p}_{_{T_0}}) + \rho_{_{T_0}}^l, \\
\Pi_h({J_{0,T_0}\mathfrak{h}}) = \mathcal{D}^{g}(u^h_{T_0}) + \rho_{_{T_0}}^h,
\end{cases}
\]
where \(\rho_{_{T_0}}^l\) and \(\rho_{_{T_0}}^h\) are given by \eqref{eq:rT0L} and \eqref{eq:rT0H}. Since \(J_{s,t}\) satisfies
\[
\partial_t J_{s,t} = -\hat{A} J_{s,t} + D\hat{F}(u_t,\mathbf{p}_t) J_{s,t}, \quad J_{s,s} = \text{Id},
\]
we can use the variation of constants formula to write its high-frequency projection as
\[
R_{s,t}^h(\Pi_h(J_{0,s}\mathfrak{h}))+\int_{s}^{t}R_{r,t}^hD_l\hat{F}_h(u_r,\mathbf{p}_{_r})\Pi_l(J_{0,r}\mathfrak{h})\, ds.
\]
This completes the proof of the lemma.
\end{proof}

We will establish quantitative estimates for \( \varrho_t = (\zeta_t, \xi_t)\). This result provides a basis for the subsequent error estimation.

\begin{lemma}\label{xi-zeta-estimate}
Fix all \( p \geq 2, \eta\in (0,1) \), there exists a constant $C:=C(N_*,\nu,\eta,\mathbf{p}_0)>0$, locally bounded in $\mathbf{p}_0$ satisfying
\begin{align}\label{eta-e}
\left( \mathbb{E} \sup_{t \in [\tau_0, T_0]} \| \varrho_t \|_{\mathcal{H} \times T_{\mathbf{p}_{_{t}}} M}^p \right)^{\frac{1}{p}} \le C\, \widetilde{T}_0^{1/2}\exp{(\eta V(u_0))}.
\end{align}
Note that \( \varrho_t \) is not adapted to the filtration \( (\mathcal{F}_t) \).
\end{lemma}
\begin{proof}
By \eqref{xi-zeta} and the Cauchy-Schwarz inequality, pathwise we have
\begin{align*}
\sup_{t\in [\tau_0,T_0]}\|\varrho_t\|&=\sup_{t\in [\tau_0,T_0]}\|\int_{\tau_0}^{t} \mathcal{J}_{s,t}\widehat{Q}_l g_s^l\,\mathrm{d}s\|
\leq
\sup_{\tau_0\leq s<t\leq T_0}\|\mathcal{J}_{s,t}\|
\int_{\tau_0}^{T_0}\left\|\hat{Q}_l g_s^l\right\|\,\mathrm{d}s\\
&\leq
C(N_*)\widetilde{T_0}^{1/2}\sup\limits_{0\le s\le t}\|\mathcal{J}_{s,t}\|_{\mathcal{H}\times T_{\mathbf{p}_{_s}}M\to \mathcal{H}\times T_{\mathbf{p}_{_t}}M}\left\|g^l\right\|_{L^2([\tau_0,T_0];\mathbb{R}^m)} .
\end{align*}
For every $\mathfrak{h}\in \mathcal{H}\times TM$ with $\|\mathfrak{h}\|=1$, taking the \(L^p(\Omega)\)-norm and using H\"older's inequality, we have 
\begin{align*}
\left(\mathbb{E}\sup_{t\in [\tau_0,T_0]}\|\varrho_t\|^p\right)^{1/p}
&\leq
C(N_*)\widetilde{T_0}^{1/2}
\left(\mathbb{E}\sup\limits_{0\le s\le t}\|\mathcal{J}_{s,t}\|_{\mathcal{H}\times T_{\mathbf{p}_{_s}}M\to \mathcal{H}\times T_{\mathbf{p}_{_t}}M}^{2p}\right)^{1/(2p)}
\left(\mathbb{E}\left\|g^l\right\|_{L^2(I)}^{2p}\right)^{1/(2p)}\\
&\leq
C\widetilde{T_0}^{1/2}e^{\eta V(u_0)}.
\end{align*}
Here the estimate of \(\mathcal{J}_{s,t}\) is used with weight \(\eta/2\), and the estimate of the control cost is also used with weight \(\eta/2\). After multiplication, this gives \(e^{\eta V(u_0)}\). By
\[
\|\zeta_t\|_l+\|\xi_t\|_h\leq C\|\varrho_t\|,
\]
we obtain \eqref{eta-e}.
\end{proof}
We now present the bounds satisfied by the error. Note that, owing to the invertibility of \(\mathcal{N}_{_{T_0}}^l\), the least-squares error arising from the construction of the control vanishes. Consequently, it suffices to consider the influence of the coupling between high and low frequencies as well as the effect of the high-frequency system itself.

\begin{proposition} \label{error2}
For every \( \eta\in (0,1) \), fixed $\gamma\in (0,1)$, there exist constants \( \widetilde{T}_0 \)  and $\tau_0$, with $0<\widetilde{T}_0\ll 1$, $\tau_0 \in [1/2,1)$ such that
\begin{align} \label{error2-eatimate}
\mathbb{E} \|\rho_{_{T_0}}\|_{\mathcal{H} \times T_{\mathbf{p}_{_{T_0}}}M}^2 \leq \, \gamma\, \exp{(\eta V(u_0))}\,\|\mathfrak{h}\|_{\mathcal{H} \times T_{\mathbf{p}_{_{T_0}}}M}^2. 
\end{align}
\end{proposition}

\begin{proof}
Using Lemma \ref{error-express}, we have that
$$
\mathbb{E}\|\rho_{_{T_0}}\|_{\mathcal{H} \times T_{\mathbf{p}_{_{T_0}}}M}^2\le \mathbb{E}\|\rho_{_{T_0}}^l\|_{\mathcal{H}_l \times T_{\mathbf{p}_{_{T_0}}}M}^2+\mathbb{E}\|\rho_{_{T_0}}^h\|_{\mathcal{H}_h }^2.
$$
Next, we proceed to estimate term by term. Regarding \( \mathbb{E}\|\rho_{_{T_0}}^l\|_{\mathcal{H} \times T_{\mathbf{p}_{_{T_0}}}M}^2 \), an application of \eqref{Rl}, \eqref{DhFl} and Lemma \ref{xi-zeta-estimate} yields 
$$
\mathbb{E}\|\rho_{_{T_0}}^l\|_{\mathcal{H} \times T_{\mathbf{p}_{_{T_0}}}M}^2 \le \widetilde{T}_0^2 e^{C\widetilde{T}_0} \exp{(\eta V(u_0))}\mathbb{E}\|\xi_t\|_{\mathbf{X}_{_{T_0}}}\le C(N_*,\nu,\mathbf{p}_0,T_0)\,\widetilde{T}_0^3\,\exp{(\eta V(u_0))}.
$$
Set 
$$
\rho_{_{T_0}}^{h,cpl}:=\int_{\tau_0}^{T_0}R_{s,T_0}^hD_l\hat{F}_h(u_s,\mathbf{p}_{_s})\zeta_s\, ds,
$$ 
then we have
\begin{align*}
\left\|\rho_{_{T_0}}^{h,cpl}\right\|_{\mathcal{H}_h}\le \left(\int_{\tau_0}^{T_0}\|R_{s,T_0}^h\|_{\mathcal{H}_h\to \mathcal{H}_h}^4\,ds\right)^{1/4}\left(\int_{\tau_0}^{T_0}\|D_l\hat{F}_h(u_s,\mathbf{p}_{_s})\|_{\mathcal{H}_l\to \mathcal{H}_h}^4\,ds\right)^{1/4}\left(\int_{\tau_0}^{T_0}\|\zeta_s\|_{\mathcal{H}_l}^2\,ds\right)^{1/2}.
\end{align*}
Furthermore, employing Lemma \ref{xi-zeta-estimate} and equations \eqref{Rh}, \eqref{DlFh}, we can obtain that
\begin{align}\label{r-hc}
\mathbb{E}\left\|\rho_{_{T_0}}^{h,cpl}\right\|_{\mathcal{H}_h}^2
&\le C(N_*,\nu) \, e^{C\widetilde{T}_0}\left (\frac{1}{4\nu N_*^2}(e^{-4\nu N_*^2\tau_0}-e^{-4\nu N_*^2T_0})\right)\,\widetilde{T}_0^{5/2}\,\exp{(\eta V(u_0))}\notag\\
&\le \frac{C(N_*,\nu)}{4\nu N_*^2} \, e^{C\widetilde{T}_0}\widetilde{T}_0^{5/2}\,\exp{(\eta V(u_0))}.
\end{align}
Set 
$$
\rho_{_{T_0}}^{h,hf}:=R_{\tau_0,T_0}^h(\Pi_h(J_{0,\tau_0}\mathfrak{h}))+\int_{\tau_0}^{T_0}R_{s,T_0}^hD_l\hat{F}_h(u_s,\mathbf{p}_{_s})\Pi_l(J_{0,s}\mathfrak{h})\, ds.
$$
Then
\begin{align*}
\left\|\rho_{_{T_0}}^{h,hf}\right\|_{\mathcal{H}_h}^2
&\le  C\|R_{\tau_0,T_0}^h\|_{\mathcal{H}_h\to \mathcal{H}_h}^2\|\Pi_h(J_{0,\tau_0}\mathfrak{h})\|_{\mathcal{H}_h}^2\\
&+ C\left(\int_{\tau_0}^{T_0} \|R_{s,T_0}^h\|_{\mathcal{H}_h\to \mathcal{H}_h}^4\, ds\right)^{1/2}\left(\int_{\tau_0}^{T_0} \|D_l\hat{F}_h\|_{\mathcal{H}_l\to \mathcal{H}_h}^4\, ds\right)^{1/2}\left(\int_{\tau_0}^{T_0} \|\Pi_l(J_{0,s}\mathfrak{h})\|_{\mathcal{H}_l}^2\, ds\right).
\end{align*}
Furthermore, using the high-frequency dissipation property \eqref{Jh} together with estimates \eqref{Rh},\eqref{DlFh} and \eqref{J}, we have 
\begin{align}
\mathbb{E}\left\|\rho_{_{T_0}}^{h,hf}\right\|_{\mathcal{H}_h}^2 
&\le Ce^{-\nu N_*^2\widetilde{T}_0} e^{C\tau_0}N_*^{-2}\,e^{\eta V(u_0)}\|\mathfrak{h}\|^2\notag\\
&\quad+ C\,e^{C\widetilde{T}_0}\left (\frac{1}{4\nu N_*^2}(e^{-4\nu N_*^2\tau_0}-e^{-4\nu N_*^2T_0})\right)\,\widetilde{T}_0^{3/2}e^{\eta V(u_0)}\|\mathfrak{h}\|^2\notag\\
&\le C\, (e^{C\tau_0}N_*^{-2}+\frac{e^{C\widetilde{T}_0}}{4\nu N_*^2}\widetilde{T}_0^{3/2})e^{\eta V(u_0)}\|\mathfrak{h}\|^2.
\end{align}
Then 
\begin{align*}
	\mathbb{E}\left\|\rho_{_{T_0}}\right\|^2 &\le C\, \Big(e^{C\tau_0}N_*^{-2}+\frac{e^{C\widetilde{T}_0}}{4\nu N_*^2}\widetilde{T}_0^{3/2}+\frac{1}{4\nu N_*^2} \, e^{C\widetilde{T}_0}\widetilde{T}_0^{5/2}+\widetilde{T}_0^{3}\Big)e^{\eta V(u_0)}\|\mathfrak{h}\|^2\\
	&:=\big(A(N_*,\nu)+B(N_*,\nu,\widetilde{T}_0)\big)e^{\eta V(u_0)}\|\mathfrak{h}\|^2,
\end{align*}
where \( A(N_*,\nu) \) vanishes as \( N_*\to \infty \). Now we can select suitable $N_*(\mathcal{E}_0,\nu)$ such that $A(N_*,\nu)\le \gamma/4$. After fixing \(N_*\), using
\[
B(N_*,\nu,\widetilde{T}_0)\to 0
\qquad \text{as}\quad \widetilde T_0 \downarrow 0,
\]
we choose \(\widetilde T_0\) sufficiently small such that $B(N_*,\nu,\widetilde{T}_0)\le \frac{\gamma}{4}$. Hence,
\[
\mathbb{E} \|\rho_{_{T_0}}\|_{\mathcal{H} \times T_{\mathbf{p}_{_{T_0}}}M}^2 \leq \, \gamma\, \exp{(\eta V(u_0))}\,\|\mathfrak{h}\|_{\mathcal{H} \times T_{\mathbf{p}_{_{T_0}}}M}^2. 
\]
\end{proof}

\section{ Nondegeneracy of the partial Malliavin matrix}\label{SEC4}
The basic idea of our proof in this section follows the classical framework of \cite{JEMS, HM06, JFA}. Since we are considering the part Malliavin matrix for a finite-dimensional system, the argument is primarily adapted from \cite{JEMS}. However, through a few technical modifications, we successfully avoid introducing the auxiliary process \( z_t \) and the truncation process \( w_t^{\rho}:=(u_t,x_t,v_t,z_t)^{\rho} \). This not only streamlines the derivation but also makes the overall proof more direct and concise.

Unlike the case of highly degenerate white noise, since \( S_t^l \) is adapted, we can use It\^{o}'s formula to write it explicitly. This allows us to relate the time derivative of certain quantities to the corresponding Lie brackets, thereby generating a new vector field on the tangent bundle.

Recall that \( \mathcal{Z}_0 := \{k\in \mathbb{Z}^2,\, 0<|k|\le N_*\} \), and  \[\hat{Q}\dot{W_t}:=\sum_{k \in \mathcal{Z}_0} q_k e_k(x) \gamma_k \, dW^k_t. \] Then, the operator \(\hat{Q}\) on \( L^2\) gives rise to a family of vector fields \(\{Q^k\}_{k \in \mathcal{Z}_0} \) on \( \mathcal{H} \times M \) defined by
\[
Q^k = 
\begin{cases} 
q_k e_k \gamma_k & \text{if } k \in \mathcal{Z}_0, \\
0 & \text{otherwise } .
\end{cases}
\]

We now present a lemma that is crucial for proving Proposition \ref{M-1}, which establishes a probabilistic spectral bound for the partial Malliavin matrix \( \mathcal{N}_{_{T_0}}^l \). 
\begin{lemma} [Probabilistic spectral bound] \label{spectral bound}
For all \( p\ge 1,\,\eta\in (0,1),\, \widetilde{T}_0 < 1, \, \varepsilon > 0 \) , any $\mathbf{p}_0\in M$, and \( (u,\mathbf{p}) \in \mathcal{H} \times M \), there exists a constant \( C(N_*,\eta, \mathbf{p}_0,p)>0 \), locally uniformly in $\mathbf{p}_0$, such that  
\begin{align}\label{spectral bound1}
\sup_{\mathfrak{h}^l \in \mathcal{H}_l \times T_{\mathbf{p}} M, \, \|\mathfrak{h}^l\|_l = 1} \mathbb{P} \left( \langle \mathcal{N}_{_{T_0}}^l \mathfrak{h}^l,\mathfrak{h}^l\rangle_l< \varepsilon \right) \le C(N_*,\eta, \mathbf{p}_0,p)\varepsilon^p e^{\eta V(u_0)},
\end{align}
where the constant is independent of \(\varepsilon\).
\end{lemma}
\begin{proof} [Proof of Proposition \ref{M-1}]
Since \( \mathcal{H}_l\times T_{\mathbf{p}}M \) is finite-dimensional with \( d = \dim(\mathcal{H}_l\times T_{\mathbf{p}}M ) \), consider a \(\delta\)-net \(\{\mathfrak{h}_1^l, \dots, \mathfrak{h}_N^l\}\) (with \( N \le C_d\delta^{-d} \)) on the unit sphere. Let \( \lambda_{\min}(\mathcal{N}_{_{T_0}}^l) \) be the smallest eigenvalue of \( \mathcal{N}_{_{T_0}}^l \). Fix $\omega\in \Omega$, whenever \( \lambda_{\min}(\mathcal{N}_{_{T_0}}^l) \le \varepsilon \) for some \( \varepsilon > 0 \), there exists \( \mathfrak{h}^l \in \mathcal{H}_l\times T_{\mathbf{p}}M \) (with $\|\mathfrak{h}^l\|=1$) satisfying
$$
\langle \mathcal{N}_{_{T_0}}^l \mathfrak{h}^l,\mathfrak{h}^l\rangle_l< \varepsilon.
$$
Select \( \mathfrak{h}_i^l \) such that \( \|\mathfrak{h}^l-\mathfrak{h}_i^l\|<\delta \), it follows that
\begin{align*}
\langle \mathcal{N}_{_{T_0}}^l \mathfrak{h}_i^l,\mathfrak{h}_i^l\rangle_l &= \langle \mathcal{N}_{_{T_0}}^l \mathfrak{h}^l,\mathfrak{h}^l\rangle_l+\langle \mathcal{N}_{_{T_0}}^l (\mathfrak{h}_i^l-\mathfrak{h}^l),\mathfrak{h}_i^l\rangle_l+\langle \mathcal{N}_{_{T_0}}^l \mathfrak{h}^l,\mathfrak{h}_i^l-\mathfrak{h}^l\rangle_l\\
& \le \langle \mathcal{N}_{_{T_0}}^l \mathfrak{h}^l,\mathfrak{h}^l\rangle_l+\|\mathcal{N}_{_{T_0}}^l\|\cdot \|\mathfrak{h}_i^l-\mathfrak{h}^l\|+\|\mathcal{N}_{_{T_0}}^l\|\cdot \|\mathfrak{h}^l\|\cdot\|\mathfrak{h}_i^l-\mathfrak{h}^l\|\\
&\le \varepsilon+2M_{{\widetilde{T}_0}}(\omega)\cdot \delta,
\end{align*}
where 
\begin{align*}
M_{{\widetilde{T}_0}}(\omega):= \widetilde{T}_0 \cdot \sup_{\tau_0\le s\le T_0} \|S_s^l\|_{\mathcal{H}_l\times T_{\mathbf{p}_{\tau_0}}M\to \mathcal{H}_l\times T_{\mathbf{p}_{_{T_0}}}M}|\hat{Q}_l|^2\ge \int_{\tau_0}^{T_0}|S_s^l\hat{Q}_l|^2\, ds\ge \|\mathcal{N}_{_{T_0}}^l\|.
\end{align*}
Fix $R>0$, set $A_R:=\{M_{{\widetilde{T}_0}}(\omega)\le R\}$. Then, we have that
$$
\mathbb{P}\left( \lambda_{\min}(\mathcal{N}_{_{T_0}}^l )< \varepsilon \right)=\mathbb{P}\left( \{\lambda_{\min}(\mathcal{N}_{_{T_0}}^l )< \varepsilon\}\cap A_R\right)+\mathbb{P}\left( \{\lambda_{\min}(\mathcal{N}_{_{T_0}}^l )< \varepsilon\}\cap A_R^c\right).
$$
On \( A_R \), setting \( \delta = \frac{\varepsilon}{4R} >0\) yields
$$
\langle \mathcal{N}_{_{T_0}}^l \mathfrak{h}_i^l,\mathfrak{h}_i^l\rangle_l \le \varepsilon+2M_{{\widetilde{T}_0}}(\omega)\cdot \delta< 2\varepsilon.
$$
From this we further obtain
\begin{align*}
\mathbb{P}\left( \{\lambda_{\min}(\mathcal{N}_{_{T_0}}^l )< \varepsilon\}\cap A_R\right) &\le \mathbb{P}\left( \{\min_i\langle\mathcal{N}_{_{T_0}}^l\mathfrak{h}^l_i,\,\mathfrak{h}^l_i\rangle< 2\varepsilon\}\cap A_R\right)\\
& \le \sum_{i=1}^N\mathbb{P}\left( \{\langle\mathcal{N}_{_{T_0}}^l\mathfrak{h}^l_i,\,\mathfrak{h}^l_i\rangle< 2\varepsilon\}\cap A_R\right)\\
&\le C \,R^d\varepsilon^{p-d} e^{\eta V(u_0)}.
\end{align*}

Next, we estimate the truncated remainder term. By Markov's inequality, for any \( q \ge 1 \), we have
$$
\mathbb{P}(A_R^c)=\mathbb{P}(M_{{\widetilde{T}_0}}(\omega)>R)\le \frac{\mathbb{E}[M_{{\widetilde{T}_0}}(\omega)^q]}{R^q}\le C \left(\frac{\widetilde{T}_0}{R}\right)^q e^{\eta V(u_0)}.
$$
Take \( R = \varepsilon^{-\alpha} \), with \(\alpha(>0)  \) to be fixed later. It follows that
$$
\mathbb{P}\left( \{\lambda_{\min}(\mathcal{N}_{_{T_0}}^l )< \varepsilon\}\right) \le C \varepsilon^{p-d-\alpha d} e^{\eta V(u_0)}+\widetilde{T}_0^q \varepsilon^{\alpha q} e^{\eta V(u_0)}.
$$
Set $\alpha=\frac{p-d}{d+q}$. Then we have $p-d-\alpha d=\alpha q$ and 
$$
\mathbb{P} \left( \{\lambda_{\min}(\mathcal{N}_{_{T_0}}^l )< \varepsilon\}\right) \le C (\widetilde{T}_0^q+1)\varepsilon^{\frac{q(p-d)}{d+p}}e^{\eta V(u_0)}.
$$
Since \( p (> d) \) and \( q \) can both be taken arbitrarily large, we set \( p = q \) to simplify the calculations. This yields
$$
\mathbb{P} \left( \{\lambda_{\min}(\mathcal{N}_{_{T_0}}^l )< \varepsilon\}\right) \le C (\widetilde{T}_0^p+1)\varepsilon^{\frac{p(p-d)}{d+p}}e^{\eta V(u_0)}.
$$
Now, set $X:=\lambda_{\min}(\mathcal{N}_{_{T_0}}^l )$, then for any \( r > 0 \), we compute \( \mathbb{E}[X^{-r}] \). Note that
\begin{align*}
\mathbb{E}[X^{-r}] &= \int_0^1 \mathbb{P}(X<a)ra^{-r-1}\,da+\int_1^{\infty}\mathbb{P}(X<a)ra^{-r-1}\,da\\
& \le C \int_0^1(\widetilde{T}_0^p+1)a^{\frac{p(p-d)}{d+p}}e^{\eta V(u_0)}ra^{-r-1}\,da.
\end{align*}
The above integral converges if and only if \( \frac{p(p-d)}{d+p}-r>0 \), which means \( r< \frac{p(p-d)}{d+p}\). By choosing \( p \) sufficiently large (\( p > r + d + 1 \)), we have
$$
\mathbb{E}[X^{-r}] \le C(N_*,\nu, \mathbf{p}_0,\eta) (\widetilde{T}_0^{p+d+1}+1)e^{\eta V(u_0)}
$$
Furthermore, when \( \widetilde{T}_0 \) is sufficiently small, it follows that
\[
\mathbb{E}|(\mathcal{N}_{_{T_0}}^l)^{-1}|^p \leq C(N_*,\nu,\mathbf{p}_{_0},\eta,p) \exp{(p\eta V(u_0))}.
\]
\end{proof}

We now prove the probabilistic spectral bounds for the partial Malliavin matrix, stated in Lemma \ref{spectral bound}. The proof is divided into two main steps. First, we establish a statement analogous to the generalized H\"{o}rmander condition (Lemma \ref{lower}). Then, by differentiating in time, we derive a series of implications (Lemma \ref{upper}); combining these two parts completes the proof. Here, the Lie brackets associated with the velocity field and the vector fields on the manifold, which appear in the generalized H\"{o}rmander condition, are naturally captured by the integral expression \(S_t^l Q^k\). Using It\^{o}'s formula, the integral expression given in \cite[Proposition 6.12]{JEMS} is as follows.
\begin{lemma} \label{SE}
Let \( E \) be a bounded vector field on \( \mathcal{H} \times M \) whose range belongs to \( \mathcal{H}_l \times TM \) and with two bounded derivatives. Then the following formula holds:
\begin{align}\label{SE2}
S_t^L E(u_t,\mathbf{p}_{_t}) &= E(u_{\tau_0},\mathbf{p}_{_{\tau_0}}) + \int_{\tau_0}^{t} S_s^l\, \bigl([\hat{F}, E]_l (u_s,\mathbf{p}_{_s}) - [\hat{A}, E]_l (u_s,\mathbf{p}_{_s})\bigr) \, ds \notag\\
&\quad + \frac{1}{2} \sum_{k \in \mathcal{Z}_0} \int_{\tau_0}^{t} S_s^l\, D^2 E(u_s,\mathbf{p}_{_s}) [Q^k, Q^k] \, ds \notag\\
&\quad + \int_{\tau_0}^{t} S_s^l\, D E(u_s,\mathbf{p}_{_s}) \hat{Q} \, dW_s,
\end{align}
Here, for any two differentiable vector fields \( F, E \) over \( \mathcal{H} \times M \), we denote
\[
[\hat{F}, E]_l := (D E_l)(u,\mathbf{p})\hat{F}(u,\mathbf{p}) - (D \hat{F}_l)(u,\mathbf{p})E(u,\mathbf{p}) = (D E)(u,\mathbf{p})\hat{F}(u,\mathbf{p}) - (D_l \hat{F}_l)(u,\mathbf{p})E(u,\mathbf{p})
\]
and
\[
[\hat{A}, E]_l (u,\mathbf{p}) := D_l E_l (u,\mathbf{p}) \hat{A}_l (u,\mathbf{p}) - \hat{A}_l E(u,\mathbf{p}) = D E (u,\mathbf{p}) \hat{A} (u,\mathbf{p}) - \hat{A}_l E(u,\mathbf{p}).
\]
\end{lemma}

For convenience, we define the following operator \( \Upsilon_l \) that maps smooth vector fields on \( \mathcal{H} \times M \) to smooth vector fields on \( \mathcal{H} \times M \) with range in \( \mathcal{H}_l \times TM \), defined by
\begin{align} \label{Upsilon}
\Upsilon_l E := [
\hat{F}, E]_l - [\hat{A}, E]_l + \frac{1}{2} \sum_{k \in \mathcal{Z}_0} D^2 E[Q^k, Q^k].
\end{align}
We first present a lower bound result, which is essentially analogous to the generalized H\"{o}rmander condition (c.f. \cite{JFA}).
\begin{lemma} \label{lower}
For any $(u_t^l,\mathbf{p}_{_t})\in \mathcal{H}_l\times M$, there exists a constant $C(N_*,\nu)>0$ such that the lower-bound
\[
\max\left\{ |\langle Q^k, \mathfrak{h}^l \rangle_l|,\; |\langle \Upsilon_l Q^k, \mathfrak{h}^l \rangle_l| : k \in \mathcal{Z}_0 \right\} \ge \frac{C(N_*,\nu)}{1+\|u^l\|_{\mathcal{H}_l}} \|\mathfrak{h}^l\|_l
\]
holds for every \( \mathfrak{h}^l \in \mathcal{H}_l \times T_{\mathbf{p}} M \).
\end{lemma}

Before proving the above lemma, we first state a {\emph{spanning condition on the tangent bundle of the manifold}}. This condition appears in different forms in the literature \cite{JEMS, NS}; here we present the version given in \cite{JEMS}. To this end, we rewrite equations \eqref{x}–\eqref{v} as follows:
\[
\frac{d}{dt} \begin{pmatrix}
x_t \\
v_t
\end{pmatrix} = \overline{F}(u_t, x_t, v_t),
\]
where \( \overline{F}(u, x, v) \) is the vector field defined for each \((u, x, v) \in \mathcal{H} \times \mathbb{T}^2 \times S^{1}\) by
\begin{align}\label{F1}
\overline{F}(u, x, v) = \sum_{k \in \mathcal{Z}_0} \begin{pmatrix}
(u)_k e_k(x) \gamma_k \\
(u)_k (k \cdot v) e_{-k}(x)(\Pi_v \gamma_k)
\end{pmatrix} \in T_x \mathbb{T}^2 \times T_v {S}^{1},
\end{align}
where \( (u)_k = \frac{1}{\pi(2\pi)^{d-1}} \langle u, e_k \gamma_k \rangle_{L^2} \). Moreover, it directly follows that
\[
[e_k \gamma_k, \overline{F}](x, v) = \begin{pmatrix} e_k(x) \gamma_k \\ (k \cdot v) e_{-k}(x)(\Pi_v \gamma_k) \end{pmatrix}.
\]
\begin{lemma}\emph{\cite[Lemma 5.3]{JEMS}} \label{TM span}
Let \( k^1, k^2 \) be linearly independent elements of \(\mathbb{Z}_0^2\) and define  
\[
K = \{k^1, k^2\} \cup \{-k^1, -k^2\} \subseteq \mathbb{Z}_0^2.
\]  
Then at each point \((x, v) \in \mathbb{T}^2 \times {S}^{1}\), we have  
\[
\operatorname{span}\{[e_k \gamma_k, \overline{F}](x, v): k \in K\} = T_x \mathbb{T}^2 \times T_v {S}^{1}.
\]
\end{lemma}

\begin{proof}[Proof of Lemma \ref{lower}]
Note that
\begin{align}\label{U-1}
\Upsilon_l Q^k=q_k[\overline{F},e_k\gamma_k]-q_k[B(u,u),e_k\gamma_k]_l-q_k[\hat{A},e_k\gamma_k]_l.
\end{align}
Set $q_{-}:=\min\limits_{k\in \mathcal{Z}_0}|q_k|>0$, then we have 
\begin{align}\label{g-1}
\max\limits_{k\in \mathcal{Z}_0}\langle e_k\gamma_k,\,\mathfrak{h}^l\rangle_l\le q_{-}^{-1}\max\limits_{k\in \mathcal{Z}_0}\left\{ |\langle Q^k, \mathfrak{h}^l \rangle_l|,\; |\langle \Upsilon_l Q^k, \mathfrak{h}^l \rangle_l|\right\}.
\end{align}
Noted that $A(e_k)=\lambda_k e_k,\, \lambda_k=\nu|k|^2$ and 
$$
[\hat{A},e_k\gamma_k]_l = D(e_k\gamma_k)\hat{A}-\hat{A}_l(e_k\gamma_k)=-\gamma_k\hat{A}_l(e_k),
$$
then we have 
\begin{align} \label{A-1}
|\langle[\hat{A},e_k\gamma_k]_l,\,\mathfrak{h}^l\rangle_l|=\lambda_k|\langle e_k\gamma_k,\,\mathfrak{h}^l\rangle_l|\le C(N_*,\nu) \,|\langle e_k\gamma_k,\,\mathfrak{h}^l\rangle_l|.
\end{align}

Simultaneously, observe that
$$
-[B(u,u),e_k\gamma_k]_l=\Pi_l[B(u,e_k\gamma_k)+B(e_k\gamma_k,u)]
$$
Let \(\{\phi_1, \ldots, \phi_{m}\}\) be an orthonormal basis of \(\mathcal{H}_l\), where each \(\phi_k\) is of the form \(e_k \gamma_k\). Since \(\mathcal{H}_l\) is finite-dimensional, the linear operator \(D(B(u, u)) : \mathcal{H}_l \to \mathcal{H}\) admits a matrix representation. Consequently, for any \(\mathfrak{h}^{l,u} \in \mathcal{H}_l\), we have
$$
\Pi_l(D(B(u, u))\cdot \mathfrak{h}^{l,u})=\sum_{k=1}^{m}a_k\,\Pi_l(D(B(u, u))\, \phi_k),\quad \text{where } \,a_k=\langle\mathfrak{h}^{l,u},\phi_k\rangle_{l,u}.
$$
Noted that $\Pi_l(D(B(u, u))\, \phi_k)\in \mathcal{H}_l$, then we have that
$$
\Pi_l(D(B(u, u))\, \phi_k) = \sum_{j=1}^{m}C_{k,j}(u)\phi_j, \quad \text{where } \,C_{k,j}(u)=\langle\Pi_l(D(B(u, u))\, \phi_k,\, \phi_j\rangle_{l}.
$$
Moreover, we obtain that 
$$
[B(u,u),e_k\gamma_k]_l=\Pi_l(D(B(u,u))\cdot \phi_k)=\sum_{j=1}^{m}C_{k,j}(u)\phi_j.
$$
We now establish a uniform estimate for $C_{k,j}(u)$ with respect to $(k,j)$. 
\begin{align*}
|C_{k,j}(u)|\le |\langle \Pi_l(B(\phi_k,u)),\,\phi_j\rangle_l|\,+\,|\langle \Pi_l(B(u,\phi_k)),\,\phi_j\rangle_l|\le C(N_*)\|u^l\|_{\mathcal{H}_l}.
\end{align*}
Therefore, 
\begin{align}\label{B-1}
|\langle[B(u,u),e_k\gamma_k]_l,\,\mathfrak{h}^l\rangle_l|\le C(N_*)\,\|u^l\|_{\mathcal{H}_l}\, \sum_{k=1}^{m}|\langle e_k\gamma_k,\,\mathfrak{h}^l \rangle_l|\le C(N_*)m\,\|u^l\|_{\mathcal{H}_l}\max\limits_{k\in \mathcal{Z}_0}|\langle e_k\gamma_k,\,\mathfrak{h}^l\rangle_l|.
\end{align}
Putting together \eqref{U-1}–\eqref{B-1} yields
\begin{align}\label{4.8}
	\max\limits_{k \in \mathcal{Z}_0} \bigl\{ \bigl| \bigl\langle e_k\gamma_k,\, \mathfrak{h}^l \bigr\rangle_l \bigr|,\,\bigl| \bigl\langle [\overline{F}, e_k \gamma_k]_l,\, \mathfrak{h}^l \bigr\rangle_l \bigr|\bigr\}\le C(N_*,\nu)\,(1+\|u^l\|_{\mathcal{H}_l})\max\limits_{k \in \mathcal{Z}_0}\left\{ |\langle Q^k, \mathfrak{h}^l \rangle_l|,\; |\langle \Upsilon_l Q^k, \mathfrak{h}^l \rangle_l| \right\}.
\end{align}
On the other hand, by Lemma \ref{TM span} (the manifold spanning condition), for any \(  \mathfrak{h}^l\in \mathcal{H}_l\times T_{\mathbf{p}}M \) , we have
\[
\max \bigl\{ \bigl| \bigl\langle e_k\gamma_k,\, \mathfrak{h}^l \bigr\rangle_l \bigr|,\,\bigl| \bigl\langle [\overline{F}, e_k \gamma_k]_l,\, \mathfrak{h}^l \bigr\rangle_l \bigr|: k \in \mathcal{Z}_0\bigr\} \ge C \|\mathfrak{h}^l\|_l >0.
\]
Then, we have 
$$
C \|\mathfrak{h}^l\|_l\le C(N_*,\nu)\,(1+\|u^l\|_{\mathcal{H}_l})\max\limits_{k \in \mathcal{Z}_0}\left\{ |\langle Q^k, \mathfrak{h}^l \rangle_l|,\; |\langle \Upsilon_l Q^k, \mathfrak{h}^l \rangle_l| \right\}.
$$
The proof of the lemma is thereby concluded.
\end{proof}

In what follows, we proceed to present a `upper bound' for 
$|\langle Q^k, \mathfrak{h}^l \rangle_l|,\, |\langle \Upsilon_l Q^k, \mathfrak{h}^l \rangle_l|$. This is achieved through the operation of differentiation in time.

\begin{lemma} \label{upper}
Fix \( \tau_0\ge 1/2,\, \widetilde{T}_0 > 0 \). There is a positive constant \( \mathcal{E}(\widetilde{T}_0) > 0\) such that the following holds. Fix any \( \eta > 0 \), then for every \( \varepsilon \in (0, \mathcal{E}(\widetilde{T}_0)) \) there are a set \( \Omega_{\varepsilon}^* \) and a constant \( C(\eta, N_*,\widetilde{T}_0,\mathbf{p}_0)>0 \), locally bounded in $\mathbf{p}_0\in M$, such that
\[
\mathbb{P}((\Omega_{\varepsilon}^*)^{c}) \leq C(\eta, N_*,\widetilde{T}_0,\mathbf{p}_0) \varepsilon e^{\eta V{(u_0)}}
\]
and on \( \Omega_{\varepsilon}^* \) one has 
\begin{align}
\langle \mathcal{N}_{_{T_0}}^l \mathfrak{h}^l, \mathfrak{h}^l \rangle_l \leq \varepsilon \|\mathfrak{h}^l|_l^2 \Longrightarrow
\begin{cases} 
\sup\limits_{k\in \mathcal{Z}_0}|\langle Q^k,\,\mathfrak{h}^l\rangle_l|\le \varepsilon^{1/8}\|\mathfrak{h}^l|_l, \\ 
\sup\limits_{k\in \mathcal{Z}_0}\sup\limits_{t\in [\tau_0,T_0]}|\langle S_t^l\Upsilon_l Q^k,\,\mathfrak{h}^l\rangle_l|\le \varepsilon^{1/{80}}\|\mathfrak{h}^l\|_l.
\end{cases}
\end{align}
which is valid for any \( \mathfrak{h}^l \in \mathcal{H}_l\times T_{\mathbf{p}}M \).
\end{lemma}

Before proceeding to the proof of Lemma \ref{upper}, we first state a very useful lemma. It was developed from \cite[Lemma 6.14]{HM11a}, and its final form is given in \cite[Lemma 6.2]{JFA}.

\begin{lemma}\emph{\cite[Lemma 6.2]{JFA}} \label{JFA-Lemma 6.2}
Fix \( \tau_0\ge 1/2,\, \widetilde{T}_0 > 0 \), \(\alpha \in (0,1]\) and an index set \( \mathcal{I} \). Consider a collection of random functions \( g_{_{\mathfrak{h}^l}} \) taking values in \( C^{1,\alpha} ([\tau_0, {T}_0]) \) and indexed by \(\mathfrak{h}^l \in \mathcal{I}\). Define, for each \(\varepsilon > 0\),
\[
\Lambda_{\varepsilon,\alpha} :=\bigcup_{\mathfrak{h}^l \in \mathcal{I}} \Lambda_{\varepsilon,\alpha}^{\mathfrak{h}^l},
\]
where
\begin{align} \label{Gamma}
\Lambda_{\varepsilon,\alpha}^{\mathfrak{h}^l} := 
\left\{
\sup_{t \in [\tau_0, {T}_0]} |g_{_{\mathfrak{h}^l}}(t)| \leq \varepsilon
\quad{\emph{\text{and}}}\quad
\sup_{t \in [\tau_0, {T}_0]} |g_{_{\mathfrak{h}^l}}'(t)| > \varepsilon^{\frac{\alpha}{2(1+\alpha)}}
\right\}. 
\end{align}
Then, there exists \(\varepsilon_0 = \varepsilon_0(\alpha, \widetilde{T}_0) \big(:=({\widetilde{T}_0}/4)^{2(1+\alpha)/(2+\alpha)}\big)> 0\) such that for every \(\varepsilon \in (0, \varepsilon_0)\),
\begin{align} \label{Gamma-2}
\mathbb{P}\bigl(\Lambda_{\varepsilon,\alpha}\bigr) \leq C\, \varepsilon\; 
\mathbb{E}\Big[ \sup\limits_{\mathfrak{h}^l \in I}\, \|g_{_{\mathfrak{h}^l}}\|_{C^{1,\alpha}([\tau_0, {T}_0])}^{\,2/\alpha} \Big].
\end{align}
\end{lemma}
\begin{remark}
Observe that
\[
\Lambda_{\varepsilon,\alpha}^c = \bigcap_{\mathfrak{h}^l \in \mathcal{I}} \left\{ \sup_{t \in [\tau_0, {T}_0]} |g_{_{\mathfrak{h}^l}}(t)| > \varepsilon\, {\emph{\text{ or }}} \sup_{t \in [\tau_0, {T}_0]} |g_{_{\mathfrak{h}^l}}'(t)| \leq \varepsilon^{\frac{\alpha}{2(1+\alpha)}} \right\}.
\]
Thus, on \(\Lambda_{\varepsilon,\alpha}^c\),
\begin{align} \label{Gamma-3}
\sup_{t \in [\tau_0, {T}_0]} |g_{_{\mathfrak{h}^l}}(t)| < \varepsilon \implies \sup_{t \in [\tau_0, {T}_0]} |g_{_{\mathfrak{h}^l}}'(t)| \leq \varepsilon^{\frac{\alpha}{2(1+\alpha)}} 
\end{align}
for every \(\mathfrak{h}^l \in \mathcal{I}\).
\end{remark}

\begin{proof}[Proof of Lemma \ref{upper}]
Our proof proceeds in two steps, consisting of two operations of time differentiation.\\
\textbf{Step 1.} For any \(\mathfrak{h}^l \in \mathcal{H}_l\times T_{\mathbf{p}}M\) with \(\|\mathfrak{h}^l\| = 1\), and any $k\in \mathcal{Z}_0)$, we define
	\[
	g_{_{\mathfrak{h}^l}}(t) := \int_{\tau_0}^{t} \langle S_s^l Q^k, \mathfrak{h}^l \rangle_l^2  ds \leq \sum_{k\in \mathcal{Z}_0} \int_{\tau_0}^{t} \langle S_s^l Q^k, \mathfrak{h}^l \rangle_l^2  ds = \langle \mathcal{N}_{_{T_0}}^l\mathfrak{h}^l,\, \mathfrak{h}^l \rangle_l.
	\]
Note that
\begin{equation*}
	g_{_{\mathfrak{h}^l}}'(t) = \langle S_t^l Q^k, \mathfrak{h}^l \rangle_l^2, 
\end{equation*}
\begin{equation*}
	g_{_{\mathfrak{h}^l}}''(t)=2\langle S_t^l Q^k, \mathfrak{h}^l\rangle_l\, \langle \partial_tS_t^l Q^k, \mathfrak{h}^l\rangle_l = -2\langle S_t^l Q^k, \mathfrak{h}^l\rangle_l\,  \langle \widetilde{L}_t^lS_{t}^lQ^k, \mathfrak{h}^l\rangle_l.
\end{equation*}
Let \(\widetilde{\Omega}_{\varepsilon} := \widetilde{\Lambda}_{\varepsilon, 1}^c\), where \(\widetilde{\Lambda}_{\varepsilon, \alpha}\) is as in \eqref{Gamma} with \(\mathcal{I} := \{\mathfrak{h}^l \in \mathcal{H}_l\times T_{\mathbf{p}}M : \|\mathfrak{h}^l\|_l = 1\}\), i.e.,
\begin{equation*}
	\widetilde{\Lambda}_{\varepsilon,1} = \bigcup_{\mathfrak{h}^l \in \mathcal{I}}  
	\left\{ \sup_{t \in [\tau_0,T_0]} |g_{_{\mathfrak{h}^l}}(t)| \leq \varepsilon \quad \text{and} \sup_{t \in [\tau_0,T_0]} |g_{_{\mathfrak{h}^l}}'(t)| > \varepsilon^{\frac{1}{4}}
	\right\}.
\end{equation*}
Then by Lemma \ref{JFA-Lemma 6.2} with \(\alpha = 1\), one has
\[
	\mathbb{P}(\widetilde{\Omega}_{\varepsilon}^c) \leq C \epsilon \mathbb{E} \bigg( \sup_{\substack{t \in [\tau_0, T_0] \\ \|\mathfrak{h}^l\|_l = 1}} \big|2\langle S_t^l Q^k, \mathfrak{h}^l\rangle_l\,  \langle \widetilde{L}_t^lS_{t}^lQ^k, \mathfrak{h}^l\rangle_l\big|^2 \bigg) .
\]
Note that 
\[ \left| \langle \widetilde{L}_t^lS_{t}^lQ^k, \mathfrak{h}^l\rangle_l\right| \leqslant \|\widetilde{L}_t^lQ^k\| \|S_{t}^l\| \leqslant \frac{1}{2} \left( C \|\widetilde{L}_t^l\|^2_{{H}^{n+2}_l\times TM\to \mathcal{H}_l\times TM} + \|S_{t}^l\|^2_{\mathcal{H}_l\times TM\to \mathcal{H}_l\times TM} \right). \]
Then by \eqref{Sl}-\eqref{Ll}, we obtain
\[
\mathbb{P}(\widetilde{\Omega}_{\varepsilon}^c) \leq C\exp(\eta V(u_0)) \varepsilon
\]
for any \(\varepsilon < \varepsilon_0(\widetilde{T}_0)\), where \(C = C(\eta, \widetilde{T}_0,\mathbf{p}_0)\) that is locally bounded in \( \mathbf{p}_0 \in M \). Finally, on \(\widetilde{\Omega}_{\varepsilon}\) we have, cf. \eqref{Gamma-2}, that
	\begin{align} \label{step1}
	\langle \mathcal{N}_{_{T_0}}^l\mathfrak{h}^l,\, \mathfrak{h}^l \rangle_l \leq \varepsilon \|\mathfrak{h}^l\|_l^2 \Longrightarrow \sup_{t \in [\tau_0, T_0]} |\langle S_t^l Q^k,\,\mathfrak{h}^l\rangle_l| \leq \varepsilon^{1/8} \|\mathfrak{h}^l\|_l,
	\end{align}
for each \(k \in \mathcal{Z}_0\) and any \(\mathfrak{h}^l \in \mathcal{H}_l\times T_{\mathbf{p}}M\).\\
\textbf{Step 2.} Using the Lemma \ref{SE}, for any $k\in \mathcal{Z}_0$, we obtain that
$$
S_t^l Q^k=Q^k+\int_{\tau_0}^{t}S_s^l \Upsilon_l Q^k\,ds.
$$
Then by \eqref{step1}, on \(\widetilde{\Omega}_{\varepsilon}\) we have that
\begin{align*} 
\langle \mathcal{N}_{_{T_0}}^l\mathfrak{h}^l,\, \mathfrak{h}^l \rangle_l \leq \varepsilon \|\mathfrak{h}^l\|_l^2 \Longrightarrow
\begin{cases} 
\sup\limits_{k\in \mathcal{Z}_0}|\langle Q^k,\,\mathfrak{h}^l\rangle_l|\le \varepsilon^{1/8}\|\mathfrak{h}^l\|_l, \\ 
\sup\limits_{k\in \mathcal{Z}_0}\sup\limits_{t\in [\tau_0,T_0]}|\int_{\tau_0}^t\langle S_s^l\Upsilon_l Q^k,\,\mathfrak{h}^l\rangle_l\, ds|\le \varepsilon^{1/{8}}\|\mathfrak{h}^l\|_l.
\end{cases}
\end{align*}
For fixed \(\mathfrak{h}^l \in \mathcal{H}_l\times T_{\mathbf{p}}M\), let 
\[
g_{_{\mathfrak{h}^l}}(t) := \int_{\tau_0}^t\langle S_s^l\Upsilon_l Q^k,\,\mathfrak{h}^l\rangle_l\, ds.
\] 
Then 
$$
g_{_{\mathfrak{h}^l}}'(t) = \langle S_t^l\Upsilon_l Q^k,\,\mathfrak{h}^l\rangle_l,\;\;g_{_{\mathfrak{h}^l}}''(t) = \partial_t\big(\langle S_t^l\Upsilon_l Q^k,\,\mathfrak{h}^l\rangle_l\big).
$$
Let \(\overline{\Omega}_{\varepsilon} = \overline{\Lambda}_{\varepsilon, 1/4}^c\) (with $\alpha=1/4$), where \(\overline{\Lambda}_{\epsilon,\alpha}\) is as in \eqref{Gamma} over with \(\mathcal{I} := \{\mathfrak{h}^l \in \mathcal{H}_l\times T_{\mathbf{p}}M : \|\mathfrak{h}^l\|_l = 1\}\). Then, on \(\overline{\Omega}_{\varepsilon}\) one has, in view of \eqref{Gamma-3},
\begin{equation}\label{SUQ}
\sup_{t \in [\tau_0,T_0]} |\int_{\tau_0}^t\langle S_s^l\Upsilon_l Q^k,\,\mathfrak{h}^l\rangle_l\, ds| \leq \varepsilon \|\mathfrak{h}^l\|_l \Longrightarrow \sup_{t \in [\tau_0,T_0]} |\langle S_t^l\Upsilon_l Q^k,\,\mathfrak{h}^l\rangle_l| \leq \epsilon^{1/10} \|\mathfrak{h}^l\|_l .
\end{equation}
By \eqref{Gamma-2}, we have that
\begin{align}\label{Gamma-4}
\mathbb{P}((\overline{\Omega}_{\varepsilon})^c) &\leq C \epsilon \mathbb{E} \left( \sup_{\mathfrak{h}^l \in \mathcal{I}} \| g_{_{\mathfrak{h}^l}}'(t) \|_{C^{1/4}([\tau_0,T_0])}^{8} \right).
\end{align}
Set $Y_t^k:= \langle S_t^l\Upsilon_l Q^k,\,\mathfrak{h}^l\rangle_l$, it now suffices to estimate $\mathbb{E} \big( \sup_{\mathfrak{h}^l \in \mathcal{I}} \|Y_t^k\|_{C^{1/4}([\tau_0,T_0])}^{8} \big)$. Noted that
\begin{align} \label{Gamma-5}
Y_t^k = Y_0^k+\int_{\tau_0}^{t}\langle S_s^l\Upsilon_l(\Upsilon_l Q^k),\, \mathfrak{h}^l\rangle_l\, ds + 
\int_{\tau_0}^{t}\langle S_s^lD(\Upsilon_l Q^k)\hat{Q}\, dW_s,\, \mathfrak{h}^l\rangle_l.  
\end{align}
Setting $\alpha^k_s:=\langle S_s^l\Upsilon_l(\Upsilon_l Q^k),\, \mathfrak{h}^l\rangle_l$, and $\gamma_s^k:=\langle S_s^lD(\Upsilon_l Q^k)\hat{Q},\, \mathfrak{h}^l\rangle_l$ as a linear operator from ${\mathbb{R}^{m}}$ to $\mathbb{R}$, we then rewrite \eqref{Gamma-5} as
$$
Y_t^k = Y_0^k+\int_{\tau_0}^{t} \alpha^k_s\, ds + \int_{\tau_0}^{t}\gamma_s^k\, dW_s.
$$

We first estimate \(\alpha_S^k\). From \eqref{Upsilon} we immediately have
$$
\Upsilon_l(\Upsilon_l Q^k) = q_k[\overline{F},\,\Upsilon_l Q^k]-q_k[B(u,u),\,\Upsilon_l Q^k]_l-q_k[\hat{A},\,\Upsilon_l Q^k]_l =: \, q_k[\hat{F},\,\Upsilon_l Q^k]_l-q_k[\hat{A},\,\Upsilon_l Q^k]_l,
$$
where
\[
\hat{F} =
\begin{pmatrix}
-B(u, u) \\
u(x) \\
\Pi_v Du(x)v 
\end{pmatrix},
\quad \Upsilon_l Q^k =
\begin{pmatrix}
-\Pi_l\big(B(u,e_k\gamma_k)+B(e_k\gamma_k,u)\big)-\hat{A}_le_k\gamma_k\\
e_k\gamma_k \\
(k\cdot v) e{-k}(x)(\Pi_v \gamma_k)
\end{pmatrix}.
\]
A direct calculation gives
\[
\Pi_l\big(D(\Upsilon_l Q^k)\hat{F}\big) = 
\begin{pmatrix}
\Pi_l\big(\nabla B(B(u))(e_k\gamma_k)\big)- \Pi_l\Big(\nabla_x\big(\nabla B(u)(e_k\gamma_k)\big)\Big)u-\hat{A}_l\nabla_xe_ku\gamma_k\\
\nabla_xe_k(x)u\gamma_k\\
(k\cdot v) \nabla e_{-k}(x) u (\Pi_v \gamma_k)+(k \Pi_v \nabla_xu(x)v)e_{-k}(x)(\Pi_v \gamma_k)
\end{pmatrix},
\]
and
\[
\begin{aligned}
&\Pi_l\big(D\hat{F}(\Upsilon_l Q^k)\big) \\
&= 
\begin{pmatrix}
\Pi_l\left(\nabla B(u)\big(\nabla V (e_k\gamma_k)\big)\right)-\nabla_xB(u)e_k\gamma_k\\
\nabla u(x)e_k\gamma_k\\
-\Pi_l\left (\nabla(\Pi_v\nabla_xu(x)v)\cdot \big(\nabla V (e_k\gamma_k)\big)\right)+\Pi_v\nabla_x^2u(x)\cdot e_k\gamma_k v + \Pi_v\nabla_xu(x)(k\cdot v)e_{-k}(x)(\Pi_v\gamma_k)
\end{pmatrix},
\end{aligned}
\]
where $B(u):=B(u,u)$, $\nabla B(u)\phi := B(u,\phi)+B(\phi,u)$ and $\nabla V (e_k\gamma_k):= \nabla B(u)(e_k\gamma_k)-\hat{A}_le_k\gamma_k$.

Noted that 
$$
[\hat{F},\,\Upsilon_l Q^k]_l = \Pi_l\big(D(\Upsilon_l Q^k)\hat{F}-D\hat{F}(\Upsilon_l Q^k)\big).
$$
Then, using the fact that $\mathcal{H}_l\times TM$ is finite-dimensional, we obtain
\begin{align*}
\mathbb{E} \sup_{\tau_0\le s < t\le T_0}|\Upsilon_l(\Upsilon_l Q^k)| \le C_1\mathbb{E}\sup_{\tau_0\le t\le T_0}|\|u^l\|_{H^{n+2}_l} \le C_2(N_*)\mathbb{E}\sup_{\tau_0\le t\le T_0}|\|u^l\|_{\mathcal{H}_l} \le C e^{\eta V(u_0)},
\end{align*}
where \(C = C(\eta, N_*,\widetilde{T}_0,\mathbf{p}_0)\) that is locally bounded in \( \mathbf{p}_0 \in M \). Moreover, we have that
\begin{align}\label{2-2}
	\mathbb{E} \big(\|\alpha_s^k\|_{L^{\infty}([\tau_0,T_0];\mathbb{R})}\big)\le C(\eta, N_*,\widetilde{T}_0,\mathbf{p}_0)e^{\eta V(u_0)}<\infty.
\end{align}

We now estimate \( \gamma_s^k \). Note that for any \( f \in {\mathbb{R}^{m}} \), \( \gamma_s^k f = \langle S_s^lD(\Upsilon_lQ^k)\hat{Q}f,\,\mathfrak{h}^l\rangle_l \), its Hilbert-Schmidt norm is given by
$$
\|\gamma_s^k\|_{\mathcal{L}^2({\mathbb{R}^{m}},\mathbb{R})} = \bigg(\sum_{m}\big|\langle S_s^lD(\Upsilon_lQ^k)\hat{Q}g_m,\,\mathfrak{h}^l\rangle_l\big|^2\bigg)^{1/2},
$$
where \( \{g_m\} \) is an orthonormal basis of \( {\mathbb{R}^{m}} \). Furthermore, we have that
\begin{align*}
\mathbb{E} \sup_{\tau_0\le s < t\le T_0} \|\gamma_s^k\|_{\mathcal{L}^2({\mathbb{R}^{m}},\mathbb{R})} &\le \mathbb{E} \sup_{\tau_0\le s < t\le T_0} \|S_s^l\|_{\mathcal{H}_l\times T_{\mathbf{p}_{_{\tau_0}}}M\to \mathcal{H}_l\times T_{\mathbf{p}_{_{s}}}M}\\
&\quad\cdot \mathbb{E} \sup_{\tau_0\le s < t\le T_0} \|D(\Upsilon_l Q^k)\|_{\mathcal{H}\times M\to \mathcal{H}_l\times T_{\mathbf{p}}M}\cdot 
\|\hat{Q}\|_{\mathcal{L}^2({\mathbb{R}^{m}},\mathcal{H}\times M)}\\
&\le C e^{\eta V(u_0)},
\end{align*}
where \(C = C(\eta, N_*,\widetilde{T}_0,\mathbf{p}_0)\) that is locally bounded in \( \mathbf{p}_0 \in M \). 

Finally, we estimate the moments of the H\"{o}lder semi-norm of \( Y_t^k \), our target quantity $\|Y_t^k\|_{C^{1/4}([\tau_0,T_0])}$. For any \( s,t\in [\tau_0,T_0] \), we have
$$
|Y_t^k-Y_s^k|\le |\int_s^t \alpha_{r}^k\, dr|+|\int_s^t \gamma_{r}^k\, dW_r|.
$$
By the Burkholder–Davis–Gundy inequality, for any $p\ge 2$, we have
$$
\mathbb{E} |\int_s^t \gamma_{r}^k\, dW_r|^p\le C(p)\mathbb{E}\Big(\int_s^t \|\gamma_{r}^k\|^2_{\mathcal{L}^2({\mathbb{R}^{m}},\mathbb{R})}\, dr\Big)^{p/2}\le C(p,\eta, N_*,\widetilde{T}_0,\mathbf{p}_0)|t-s|^{p/2}e^{\eta V(u_0)}.
$$
Moreover, we obtain that
$$
\mathbb{E} |Y_t^k-Y_s^k|^p\le C(p,\eta, N_*,\widetilde{T}_0,\mathbf{p}_0) |t-s|^{p/2}e^{\eta V(u_0)}\le C|t-s|^{p/2}.
$$
Then, by the Kolmogorov continuity theorem, we know that
$$
\mathbb{E} \Big(\sup_{\mathfrak{h}^l\in \mathcal{I}}\|Y^k_t\|_{C^{1/4}([\tau_0,T_0])}^8\Big) \le C e^{\eta V(u_0)},
$$
where \(C = C(\eta, N_*,\widetilde{T}_0,\mathbf{p}_0)\) that is locally bounded in \( \mathbf{p}_0 \in M \). 
Combining with \eqref{Gamma-4}, we have
\begin{align*}
\mathbb{P}((\overline{\Omega}_{\varepsilon})^c) &\leq C(\eta, N_*,\widetilde{T}_0,\mathbf{p}_0) \epsilon e^{\eta V(u_0)}.
\end{align*}
Combining \textbf{Step 1}, setting $\mathcal{E}(\widetilde{T}_0):=\min\big\{(\frac{\widetilde{T}_0}{4})^{10/9},(\frac{\widetilde{T}_0}{4})^{4/3}\big\}$. Then for every \( \varepsilon \in (0, \mathcal{E}(\widetilde{T}_0)) \), setting $\Omega^*_{\varepsilon}:=\overline{\Omega}_{\varepsilon} \bigcap \widetilde{\Omega}_{\varepsilon}$, we have on $\Omega^*_{\varepsilon}$ that
\begin{align*}
\langle \mathcal{N}_{_{T_0}}^l \mathfrak{h}^l, \mathfrak{h}^l \rangle_l \leq \varepsilon \|\mathfrak{h}^l|_l^2 \Longrightarrow
\begin{cases} 
\sup\limits_{k\in \mathcal{Z}_0}|\langle Q^k,\,\mathfrak{h}^l\rangle_l|\le \varepsilon^{1/8}\|\mathfrak{h}^l|_l, \\ 
\sup\limits_{k\in \mathcal{Z}_0}\sup\limits_{t\in [\tau_0,T_0]}|\langle S_t^l\Upsilon_l Q^k,\,\mathfrak{h}^l\rangle_l|\le \varepsilon^{1/{80}}\|\mathfrak{h}^l\|_l,
\end{cases}
\end{align*}
and there exist a constant \(C(\eta, N_*,\widetilde{T}_0,\mathbf{p}_0)>0\) (locally bounded in \( \mathbf{p}_0 \in M \)) such that
\[
\mathbb{P}((\Omega_{\varepsilon}^*)^{c}) \leq C(\eta, N_*,\widetilde{T}_0,\mathbf{p}_0) \varepsilon e^{\eta V({u_0})}.
\]
\end{proof}

The proof of Lemma \ref{spectral bound} now follows straightforwardly.
\begin{proof}[Proof of Lemma \ref{spectral bound}]
Taking \( t = \tau_0 \) in Lemma \ref{upper}, then for any \( \varepsilon \in (0,\mathcal{E}_0) \), on the set \( \Omega_{\varepsilon}^* \) we have
$$
\langle \mathcal{N}_{_{T_0}}^l \mathfrak{h}^l, \mathfrak{h}^l \rangle_l \leq \varepsilon \|\mathfrak{h}^l|_l^2 \Longrightarrow
\max\Big\{|\langle Q^k,\,\mathfrak{h}^l\rangle_l|,\,|\langle \Upsilon_l^{\tau_0} Q^k,\,\mathfrak{h}^l\rangle_l|:\, k\in \mathcal{Z}_0\Big\}\le \varepsilon^{1/{80}}\|\mathfrak{h}^l\|_l.
$$
Furthermore, combining with Lemma \ref{lower}, we obtain on \( \Omega_{\varepsilon}^*  \) that
\begin{align*}
\langle \mathcal{N}_{_{T_0}}^l \mathfrak{h}^l, \mathfrak{h}^l \rangle_l \leq \varepsilon  \Longrightarrow \frac{C(N_*,\nu)}{1+\|u^l\|_{\mathcal{H}_l}} \|\mathfrak{h}^l\|_l\le \max\Big\{|\langle Q^k,\,\mathfrak{h}^l\rangle_l|,\,|\langle \Upsilon_l^{\tau_0} Q^k,\,\mathfrak{h}^l\rangle_l|:\, k\in \mathcal{Z}_0\Big\}\le \varepsilon^{1/{80}}.
\end{align*}
Then,
$$
\left\{
\left\langle \mathcal{N}_{T_0}^{l} h^{l}, h^{l} \right\rangle_l
< \varepsilon
\right\}
\subset
\left(\Omega_{\varepsilon,p}^{*}\right)^c
\cup
\left\{
\left\|u_{\tau_0}^{l}\right\|_{\mathcal{H}_l}
\ge c_{*}\varepsilon^{-1/80}-1
\right\}.
$$
Note that, by Markov's inequality,
\begin{equation}
	\begin{aligned}
		\mathbb{P}\left(
		\left\|u_{\tau_0}^{l}\right\|_{\mathcal{H}_l}
		\ge c_{*}\varepsilon^{-1/80}-1
		\right)\le
		C_m \varepsilon^{m/80}
		\mathbb{E}
		\left|u_{\tau_0}^{l}\right|_{\mathcal{H}_l}^{m}\le
		C_{m,\eta,N_{*}}
		\varepsilon^{m/80}
		e^{\eta V(u_0)} .
	\end{aligned}
	\tag{18}
\end{equation}
Take $m\ge 80p$. Then there exist a constant \( C(N_*,\eta, \mathbf{p}_0,p)>0 \), locally uniformly in $\mathbf{p}_0$, such that for any \( \varepsilon\in (0,{\mathcal{E}}_0) \), we have
\[
\sup_{\mathfrak{h}^l \in \mathcal{H}_l \times T_{\mathbf{p}} M, \, \|\mathfrak{h}^l\| = 1} \mathbb{P} \left( \langle \mathcal{N}_{_{T_0}}^l \mathfrak{h}^l,\mathfrak{h}^l\rangle_l< \varepsilon \right) \le C(N_*,\eta, \mathbf{p}_0,p)\varepsilon^p e^{\eta V(u_0)}.
\]
\end{proof}

\appendix
\section{Basic estimates}\label{priori}
This section first collects several basic tools used throughout the paper: moment estimates for the solution, the super-Lyapunov function and its key properties. We then present Jacobian derivative estimates that will be repeatedly invoked in Sections \ref{SEC3} and \ref{SEC4}.

The moment estimates for the baseline process \(u_t\) in this section are obtained by classical arguments; the corresponding results and full proofs can be found in \cite[Appendix A.1]{NS}. To avoid repetition, we only state the results needed here.
\begin{proposition}\label{ME}
Suppose that \(u_{t}\) is a solution to \eqref{2NS}. Then there exists \(C>0\) such that for all \(\eta\leqslant C^{-1}\) and for all \(0\leqslant s\leqslant t\),
\begin{align} \label{NS-2.4}
\mathbb{E}\exp\Bigl(\eta\sup_{s\leqslant r\leqslant t}\Bigl(\|\nabla u_{r}\|^{2}_{L^{2}}+\nu\int_{s}^{r}\|u_{a}\|^{2}_{H^{2}}\,\mathrm{d}a\Bigr)\Bigr) \leqslant Ce^{C(t-s)}\exp\bigl(e^{-C^{-1}s}\eta\|\nabla u_{0}\|^{2}_{L^{2}}\bigr). 
\end{align}
Additionally, there exists \(C(n)>0\) such that for all \(0\leqslant\eta\leqslant C^{-1}\) and \(0\leqslant s\leqslant t\),
\begin{align} \label{NS-2.5}
\mathbb{E}\exp\Bigl(\eta\sup_{s\leqslant r\leqslant t}\Bigl(\|u_{r}\|^{2}_{H^{n+1}}+\nu\int_{s}^{r}\|u_{a}\|^{2}_{H^{n+2}}\,\mathrm{d}a\Bigr)^{\frac{1}{n+2}}\Bigr)\leqslant Ce^{Ct}\exp\Bigl(e^{-C^{-1}s}\eta\| u_{0}\|^{\frac{2}{n+2}}_{H^{n+1}}+C\eta\|\nabla u_{0}\|^{2}_{L^{2}}\Bigr), 
\end{align}
and
\begin{align} \label{NS-2.6}
\mathbb{E}\exp\Bigl(\eta\sup_{s\leqslant r\leqslant t}\Bigl(\|u_{r}\|^{2}_{H^{n+1}}+\nu\int_{s}^{r}\|u_{a}\|^{2}_{H^{n+2}}\,\mathrm{d}a\Bigr)^{\frac{1}{n+2}}\Bigr)\leqslant Ce^{Cs^{-1}}e^{Ct}\exp\Bigl(C\eta\|\nabla u_{0}\|^{2}_{L^{2}}\Bigr). 
\end{align}
\end{proposition}

\begin{definition}[Super-Lyapunov function for the base process] \label{superL}
For \(u\in H_0^{5}(\mathbb{T}^{2},\mathbb{R}^2)\), we define
\begin{align}\label{SL1}
V(u):=\sigma\bigl(\|u\|^{2}_{H_0^{1}}+\alpha\|u\|^{1/3}_{H_0^{5}}\bigr),
\end{align}
where \(\alpha,\sigma>0\) are some fixed small constants, chosen according to Proposition \ref{ME} so that Corollary \ref{SLP} holds for all \(\eta\in(0,2)\).
\end{definition}

\begin{corollary}[Super-Lyapunov property] \label{SLP}
There exist \(\delta\in(0,1)\) and a constant \(C>0\) such that for all \(\eta\in(0,2)\) and \(t\geqslant 1\),
\[
\mathbb{E}\exp(\eta V(u_{t}))\leqslant C\exp(\delta\eta V(u_{0})).
\]
\end{corollary}
Now, we provide several Jacobian estimates needed in this paper.
\begin{lemma}
For any $q>0$, $\eta\in(0,1)$ there is $C(\mathbf{p}_0,\eta, q,N_*)>0$, locally bounded in $\mathbf{p}_0\in M$, such that
\begin{align}
&\mathbb{E} \sup\limits_{0\le s\le t}\|R_{s,t}^h\|^2_{\mathcal{H}_h\to \mathcal{H}_h} \le C e^{(C-\nu N_*^2)(t-s)}\exp(\eta V(u_0)),\label{Rh}\\
&\mathbb{E} \sup\limits_{0\le s\le t}\|R_{s,t}^l\|^q_{\mathcal{H}_l\times T_{\mathbf{p}_{_s}}M\to \mathcal{H}_l\times T_{\mathbf{p}_{_t}}M}\le C e^{C(t-s)}\exp(\eta V(u_0)),\label{Rl}\\
&\mathbb{E}\sup\limits_{0\le s\le t}\|S_{s,t}^l\|^q_{\mathcal{H}_l\times T_{\mathbf{p}_{_s}}M\to \mathcal{H}_l\times T_{\mathbf{p}_{_t}}M} \le C e^{C(t-s)}\exp(\eta V(u_0)),\label{Sl}\\
&\mathbb{E} \sup\limits_{\tau_0\le t\le T_0}\|\widetilde{L}_{t}^l\|^q_{{H}_l^{n+2}\times T_{\mathbf{p}_{_t}}M\to \mathcal{H}_l\times T_{\mathbf{p}_{_t}}M}\le C e^{C\widetilde{T}_0}\exp(\eta V(u_0)).\label{Ll}
\end{align}
\end{lemma}
\begin{proof}
By making slight modifications and refinements to the proof of \cite[(5.2),(5.5)]{NS}, we obtain the validity of \eqref{Rl} and \eqref{Ll}. Furthermore, note that \( S_{s,t}^l \) satisfies
\begin{align*} 
\partial_t S_{s,t}^l = \hat{A}_l S_{s,t}^l - D_l \hat{F}_l (u_t,\mathbf{p}_t) S_{s,t}^l, \quad S_{s,s}^l = \text{Id}.
\end{align*}
Hence, \eqref{Sl} follows in a similar manner. It now remains to prove \eqref{Rh}. First, observe that for any $\mathfrak{h}^h\in \mathcal{H}_h$, \( R_{s,t}^h\mathfrak{h}^h \) satisfies
\begin{align*} 
\partial_t R_{s,t}^h\mathfrak{h}^h = \nu \Pi_h(\Delta R_{s,t}^h\mathfrak{h}^h)-\Pi_h(R_{s,t}^h\mathfrak{h}^h\cdot \nabla u_t+u_t\cdot \nabla(R_{s,t}^h\mathfrak{h}^h)), \quad R_{s,s}^h\mathfrak{h}^h = \mathfrak{h}^h.
\end{align*}
Then, 
\begin{align*}
\partial_t \|R_{s,t}^h\mathfrak{h}^h\|_{\mathcal{H}_h}^2 &= 2\langle R_{s,t}^h\mathfrak{h}^h,\, \partial_t R_{s,t}^h\mathfrak{h}^h\rangle_h \\
&= -2\nu \|\nabla(R_{s,t}^h\mathfrak{h}^h)\|_{\mathcal{H}_h}^2 - 2\big\langle R_{s,t}^h\mathfrak{h}^h,\, \Pi_h\big(R_{s,t}^h\mathfrak{h}^h\cdot \nabla u_t+u_t\cdot \nabla R_{s,t}^h\mathfrak{h}^h\big)\big\rangle_h.
\end{align*}
Noted that 
\begin{align*}
&2\big\langle R_{s,t}^h\mathfrak{h}^h,\, \Pi_h\big(u_t\cdot \nabla R_{s,t}^h\mathfrak{h}^h+ R_{s,t}^h\mathfrak{h}^h\cdot \nabla u_t\big)\big\rangle_h \\
&= 2\Pi_h \bigg\langle \nabla^nR_{s,t}^h\mathfrak{h}^h,\, \sum_{j=0}^n \binom{n}{j} \left( \nabla^{n-j} \nabla^\perp \Delta^{-1} w_t \cdot \nabla \nabla^j (R_{s,t}^h\mathfrak{h}^h) + \nabla^\perp \Delta^{-1} \nabla^j (R_{s,t}^h\mathfrak{h}^h) \cdot \nabla \nabla^{n-j} w_t \right)\bigg\rangle_{L^2}.
\end{align*}
Using \cite[Lemma A.8]{NS}, a direct calculation gives
\begin{align*}
&\Pi_h\left| \int \nabla^n (R_{s,t}^h\mathfrak{h}^h) : \sum_{j=0}^n \binom{n}{j} \left( \nabla^{n-j} \nabla^\perp \Delta^{-1} w_t \cdot \nabla \nabla^j (R_{s,t}^h\mathfrak{h}^h) + \nabla^\perp \Delta^{-1} \nabla^j (R_{s,t}^h\mathfrak{h}^h) \cdot \nabla \nabla^{n-j} w_t \right) \mathrm{d}x \right|\\
&\leq \frac{\nu}{2} \| R_{s,t}^h\mathfrak{h}^h\|_{H^{n+1}_h}^2 + C \| w_t \|_{H^n}^{2n+2} \| R_{s,t}^h\mathfrak{h}^h \|_{L^2_h}^2\leq \frac{\nu}{2} \| R_{s,t}^h\mathfrak{h}^h\|_{H^{n+1}_h}^2 + C \| w_t \|_{H^n}^{2n+2} \| R_{s,t}^h\mathfrak{h}^h \|_{H^n_h}^2.
\end{align*}
Thus, we have that
\begin{align*}
\partial_t \|R_{s,t}^h\mathfrak{h}^h\|_{H^n}^2 \le -\nu \|R_{s,t}^h\mathfrak{h}^h\|_{H^{n+1}}^2+C\| w_t \|_{H^n}^{2n+2} \| R_{s,t}^h\mathfrak{h}^h \|_{H^n}^2.
\end{align*}
Owing to the fact that $R_{s,t}^h\mathfrak{h}^h$ comprises exclusively high modes and contains all unstable direction, the following holds:
$$
\|R_{s,t}^h\mathfrak{h}^h\|_{H^{n+1}}^2 \ge N_*^2\|R_{s,t}^h\mathfrak{h}^h\|_{H^{n}}^2.
$$
Then, 
$$
\partial_t \|R_{s,t}^h\mathfrak{h}^h\|_{H^n}^2 \le -\big(\nu N_*^2-C\| w_t \|_{H^n}^{2n+2}\big) \|R_{s,t}^h\mathfrak{h}^h\|_{H^{n}}^2.
$$
Together, Gr\"{o}nwall's inequality and \eqref{NS-2.5} yield inequality \eqref{Rh}.
\end{proof}

In addition to the estimates for the Jacobian derivatives of both low and high frequencies, we also need some estimates for the full linear operator.
\begin{lemma}
For any $q>0$, $\eta\in(0,1)$ there are $C(\mathbf{p}_0,\eta, q,N_*)>0$, locally bounded in $\mathbf{p}_0\in M$, and $\beta:=\beta(N_*)>0$ such that
\begin{align}
&\mathbb{E} \sup\limits_{0\le s\le t}\|J_{s,t}\|^q_{\mathcal{H}\times T_{\mathbf{p}_{_s}}M\to \mathcal{H}\times T_{\mathbf{p}_{_t}}M}\le C e^{C(t-s)}\exp(\eta V(u_0)),\label{J}\\
&\mathbb{E} \|\Pi_h(J_{0,\tau_0}\mathfrak{h})\|^q_{\mathcal{H}\times T_{\mathbf{p}_{_{\tau_0}}}M}\le C e^{C\tau_0}N_*^{-q}\exp(\eta V(u_0)).\label{Jh}
\end{align}
\end{lemma}

\begin{proof}
By making refinements to the proof of \cite[(5.2)]{NS}, we obtain the validity of \eqref{J}. At the same time, we observe that 
$$
\|\Pi_hJ_{0,\tau_0}\|_{H^n\times T_{\mathbf{p}}M\to H^n\times T_{\mathbf{p}}M}^q\le N_*^{-q}\|J_{0,\tau_0}\|_{L^2\times T_{\mathbf{p}}M\to H^{n+1}\times T_{\mathbf{p}}M}^q.
$$
Then making refinements to the proof of \cite[(5.2)-(5.3)]{NS}, we obtain the bound on $\|\Pi_hJ_{0,\tau_0}\mathfrak{h}\|_{H^n\times T_{\mathbf{p}}M}^q$.
\end{proof}
Finally, we also need estimates for the low and high-frequency derivatives of the drift term $\hat{F}$, which will be used to estimate the low-high frequency coupling terms.
\begin{lemma}
For any $q>0$, $\eta\in(0,1)$ there is $C(\mathbf{p}_0,\eta, q,N_*)>0$, locally bounded in $\mathbf{p}_0\in M$ such that
\begin{align}
&\mathbb{E} \sup\limits_{\tau_0\le t\le T_0}\|D_h\hat{F}_l\|^q_{\mathcal{H}_h\to \mathcal{H}_l\times T_{\mathbf{p}}M}\le C e^{C\widetilde{T}_0}\exp(\eta V(u_0)),\label{DhFl}\\
&\mathbb{E} \sup\limits_{\tau_0\le t\le T_0}\|D_l\hat{F}_h\|^q_{\mathcal{H}_l\times T_{\mathbf{p}}M \to \mathcal{H}_h}\le C e^{C\widetilde{T}_0}\exp(\eta V(u_0)).\label{DlFh}
\end{align}
\end{lemma}

\begin{proof}
For any $\mathfrak{h}^h\in \mathcal{H}_h$, we have that 
\begin{align*}
D_h\big(-\Pi_lB(u,u)\big)\cdot\mathfrak{h}^h &= -\Pi_l\big(B(u,\mathfrak{h}^h)+B(\mathfrak{h}^h,u)\big),\\
D_hu(x)\cdot \mathfrak{h}^h &= \mathfrak{h}^h,\\ D_h(\Pi_v Du(x)v)\cdot \mathfrak{h}^h &= \Pi_v D\mathfrak{h}^h(x)v.
\end{align*}
Note that here we take \( n = 5 \), and after projecting via \( \Pi_l \) onto the finite-dimensional subspace, all norms are equivalent. Combining with \eqref{NS-2.5}, we have 
\begin{align*}
\mathbb{E} \sup_{\tau_0\le t\le T_0}\|\Pi_l \big(D_hB(u,u)\mathfrak{h}^h)\|_{\mathcal{H}_l}&\le C e^{C\widetilde{T}_0}\exp(\eta V(u_0)) \|\mathfrak{h}^h\|_h,\\
|D_hu(x)\cdot \mathfrak{h}^h|&\le \|\mathfrak{h}^h\|_{C^0} \le C\|\mathfrak{h}^h\|_h,\\
|D\mathfrak{h}^h(x)v|&\le  C\|\mathfrak{h}^h\|_h.
\end{align*}
We have thereby established \eqref{DhFl}. Similarly, \eqref{DlFh} can be proved.
\end{proof}

  
  \section*{Declarations}
  
  \noindent{Availability of data:} No new data were generated or analysed in support of this search.
  
  \noindent{Conflict of interests:} 
  The authors declare that there are no conflict of interests, we do not have any possible conflicts of interest.
  
  \noindent{Funding:} The manuscript is supported by National Natural Science Foundation of China (No. 12571189).


\end{document}